\documentclass[11pt,twoside]{preprint}
\usepackage{comment}
\usepackage{hyperref}
\usepackage{breakurl}
\usepackage{times}
\usepackage{microtype}
\usepackage{booktabs}

\usepackage{amsmath}
\makeatletter
\let\over\@@over
\let\atop\@@atop
\makeatother
\usepackage{amssymb}
\usepackage{mathrsfs}
\usepackage{mhequ}
\usepackage{mhenvs}
\usepackage{mhsymb}

\usepackage{enumitem}

\definecolor{darkgreen}{rgb}{0.1,0.7,0.1}
\definecolor{darkred}{rgb}{0.7,0.1,0.1}
\hypersetup{
citecolor=blue,
colorlinks=true, 
urlcolor= blue, 
linkcolor= black, 
bookmarksopen=true, 
pdftitle={Reconstruction}, 
}

\newtheorem{assumption}[lemma]{Assumption}

\newop{diag}

\newcommand\minus{%
  \setbox0=\hbox{-}%
  \vcenter{%
    \hrule width\wd0 height \the\fontdimen8\textfont3%
  }%
}

\makeatother

\newcommand{\aver}[2]{\bar{#1}^{\,^{#2}}}

\def\m{\mathfrak{m}}

\def\s{\mathfrak{s}}

\def\${|\!|\!|}

\def\m{\mathfrak{m}}

\def\BB{\mathscr{B}}

\def\tun{\mathbf{1}}

\def\qed{ \hfill $\square$}

\newcommand{\cA}{\mathcal{A}}
\newcommand{\cB}{\mathcal{B}}
\newcommand{\cC}{\mathcal{C}}
\newcommand{\cD}{\mathcal{D}}
\newcommand{\cE}{\mathcal{E}}

\newcommand{\cG}{\mathcal{G}}

\newcommand{\cI}{\mathcal{I}}

\newcommand{\cP}{\mathcal{P}}
\newcommand{\cQ}{\mathcal{Q}}
\newcommand{\cR}{\mathcal{R}}

\newcommand{\cT}{\mathcal{T}}

\begin{document}

\title{The reconstruction theorem in Besov spaces}
\author{Martin Hairer$^1$ and Cyril Labb\'e$^2$}
\institute{University of Warwick, \email{M.Hairer@Warwick.ac.uk}
\and Universit\'e Paris Dauphine, \email{Labbe@Ceremade.Dauphine.fr}}

\date{\today}

\maketitle

\begin{abstract}
The theory of regularity structures~\cite{Hairer2014} sets up an abstract framework of 
\textit{modelled distributions} generalising the usual
H\"older functions and allowing one to give a meaning to several ill-posed stochastic 
PDEs. A key result in that theory is the so-called reconstruction theorem: it defines a continuous 
linear operator that maps spaces of modelled distributions into the 
usual space of distributions. In the present paper, we extend the scope of this 
theorem to analogues to the whole class of Besov spaces $\cB^\gamma_{p,q}$ with non-integer regularity indices. We then show that these spaces behave very much like their classical counterparts by
obtaining the corresponding embedding theorems and Schauder-type estimates.
\end{abstract}

\tableofcontents

\section{Introduction}

The theory of regularity structures~\cite{Hairer2014} provides an analytic framework which 
turns out to be powerful in providing solution theories to classes of 
singular parabolic stochastic PDEs. An important aspect of the theory is that, instead of describing the solution 
as an element of one of the classical spaces of functions/distributions, 
one provides a local description thereof as \textit{generalised} Taylor polynomials attached
to every space-time point. 
In the special case of smooth functions, this simply corresponds to
Whitney's \cite{Whitney} interpretation of a H\"older function as the corresponding collection 
of usual Taylor polynomials associated to it. In the setting of stochastic PDEs, it is
helpful to enrich the collection of 
usual monomials with some appropriate functionals built from the driving noise.
Then, the solution of the stochastic PDE can (under some assumptions, of course) locally be expanded 
on this enlarged basis of monomials. In the case of some ill-posed stochastic PDEs, this
procedure, or some closely related procedure as in \cite{Max},
is already required to give a rigorous interpretation of what one even \textit{means} for a
(random) function/distribution to be a solution to the equation. We provide a more detailed presentation of the theory at the end of this introduction.

The original framework of the theory~\cite{Hairer2014} used direct analogues to 
H\"older spaces of functions, but it turns out that this can be generalised to the 
whole class of Besov spaces and this is the main purpose of the present work.
One motivation for this generalisation arose in a recent work on the construction of the solution of multiplicative stochastic heat equations starting from a Dirac mass at time $0$, see~\cite{mSHE}. Therein, we adapted the theory of regularity structures to $\cB^\alpha_{p,\infty}$-like spaces in order to start the equation from this specific initial condition. Indeed, while the Dirac mass in $\R^d$ has (optimal) regularity index $-d$ in H\"older spaces of distributions, it also belongs to $\cB^{-d+d/p}_{p,\infty}$ for all $p\in [1,\infty]$: this improved regularity makes the analysis 
of the PDE much simpler when working in Besov spaces.\\
Another motivation comes from Malliavin calculus. Indeed, for proving Malliavin differentiability of the solution of an SPDE, one first constructs the solution of the equation driven by a noise $\xi$ shifted in the directions of its Cameron-Martin space - typically an $L^2$ space. To that end, one can enlarge the regularity structure to include abstract monomials associated to the shift: we point out the recent work~\cite{MalliavinReg} of Cannizzaro, Friz and Gassiat on the generalised parabolic Anderson model in dimension $2$. In the case where the SPDE is additive in the noise, an alternative approach consists in lifting the shift into the polynomial regularity structure: the natural framework would then be given by $\cB_{2,2}$-type spaces of modelled distributions.\\
We also mention the very recent work 
of Pr\"omel and Teichmann~\cite{DavidJosef} where the analytical framework of the theory of regularity structures is adapted to $\cB^\gamma_{p,p}$-type spaces. We now present in more details the definitions and results obtained in the present article.

Although there is no canonical choice for the space of \textit{modelled distributions} $\cD^\gamma_{p,q}$ that would mimic the Besov space $\cB^\gamma_{p,q}$, we opt for a definition as close as possible--at least formally--to the definition of classical Besov spaces via differences, see Definition \ref{Def:Dgamma}. Then, the main results of this paper are twofold. On the one hand, we construct a ``consistent" continuous linear map from $\cD^\gamma_{p,q}$ to $\cB^{\bar\alpha}_{p,q}$
for a suitable value of $\bar\alpha$: this is what is called a \textit{reconstruction theorem}, see Theorem \ref{Th:Reconstruction}. On the other hand, we establish continuous embeddings between the spaces $\cD^\gamma_{p,q}$, see Theorem \ref{Th:Embedding}: these embeddings are the analogues of the embedding theorems that the classical Besov spaces enjoy.

The proof of the reconstruction theorem follows from similar arguments to those developed 
in~\cite{Hairer2014}. Let us recall here that the reconstruction theorem was inspired by the 
\textit{sewing lemma} of Gubinelli~\cite{Gub04}.

The embedding theorems are more delicate. Even for classical Besov spaces, their proofs are rather sensitive to the definition one chooses. In particular, they are immediate if one relies on a countable characterisation of Besov spaces (for instance, via a wavelet analysis or via the Littlewood-Paley decomposition): in that case the embedding theorems are consequences of continuous embeddings of $\ell^p$-type spaces. On the other hand, the proofs become involved if one starts from the definition via differences, see for instance Adams~\cite[Chap.V]{Adams}, Di Nezza et al.~\cite[Thm 8.2]{Hitchhiker} or Nirenberg~\cite{Nirenberg}. As we said above, our definition of modelled distributions is in the spirit of the definition of Besov spaces via differences, thus one does not expect the proofs of the embeddings to be any simpler than in the classical case. Furthermore, our setting is more complex since we are dealing not only with classical monomials, but we also 
allow for very rough functions/distributions to be represented by some of the basis vectors in our regularity structure.

The main trick that spares us technical arguments in the proofs is the following intermediate result, which may be of independent interest. If one performs suitable averages of a modelled distribution over balls of radius $2^{-n}$ centred at the points of a discrete grid, for every $n\geq 0$, then one 
obtains a countable norm which is closer in spirit to the wavelet characterisation of Besov spaces, see Definition \ref{Def:Dgammabar}. At the level of these spaces of local averages, the embedding theorems are simple to prove. Then, the key result is the equivalence between the two spaces, which is 
obtained in Theorem \ref{Th:EqDgamma}.

Let us finally mention that the embedding theorems obtained in this article allow to show that the solution to the parabolic Anderson model
\begin{equ}\label{Eq:PAM}
\partial_t u = \Delta u + u\cdot\xi\;,\quad x\in \R^3\;,
\end{equ}
(where $\xi$ is a white noise in space) that we obtained in~\cite{mSHE} is actually H\"older continuous with index $1/2-$ as a function of the spatial variable, while it was only shown therein that it belongs (locally in space) to the Besov space $\cB^{1/2-}_{p,\infty}$ with $p$ close to $1$. We refer to Subsection \ref{Subsec:PAM} for more details.

\subsection*{The theory of regularity structures in a nutshell}

In the theory of regularity structures, one describes a function/distribution locally through its collection 
of generalised Taylor expansions up to a certain maximal degree, say $\gamma$. When dealing with a smooth 
function $F:\R^d \rightarrow \R$, this simply corresponds to the collection of polynomials:
$$ f(x_0) = \sum_{k\in \N^d: |k| < \gamma} \frac1{k!} \partial_x^k F(x_0)\, X^k\;,\quad x_0 \in \R^d\;,$$
where $|k| = \sum_{i=1}^d k_i$ is the degree of the monomial $X^k = \prod_{i=1}^k X_i^{k_i}$ and $x_0$ is the base point at which the expansion is taken. The coefficients $f_k(x_0) := \frac1{k!} \partial_x^k F(x_0)$ satisfy some analytical bounds which, in particular, reflect the consistency of the expansions taken at different base points $x_0$ and $y_0$. For instance, if $F$ is a H\"older function with index $\gamma$ then one has:
\begin{equ}\label{Eq:BoundDg}
\Big| f_\ell(x_0) - \sum_{k\in \N^d: |k+\ell| < \gamma} \binom{k+\ell}{k}(x_0-y_0)^k f_{k+\ell}(y_0) \Big| \lesssim |x_0-y_0|^{\gamma-|\ell|}\;.
\end{equ}
One can thus define the space $\cD^\gamma_{\infty,\infty}$ of coefficients $x_0 \mapsto (f_k(x_0); |k| < \gamma)$ 
satisfying \eqref{Eq:BoundDg} and show that this space is in continuous bijection with the classical space 
of $\gamma$-H\"older functions on $\R^d$. Note in particular that \eqref{Eq:BoundDg} \textit{forces} the coefficients
$f_k$ to be of the form $\frac1{k!} \partial_x^k F(x_0)$ for some $F \in \CC^\gamma$, this does not have to be
assumed a priori.

When solving singular stochastic PDEs, one does not expect the solution to have much regularity so that 
a description 
as a collection of \textit{classical} Taylor expansions would not be of much use. For instance, 
the solution to the $1$-dimensional linear stochastic heat equation driven by a space-time white noise is 
H\"older-${1\over 2}^-$ in space so that no information can be gained by considering Taylor expansions 
of order greater than zero. 
The key idea is then the 
following: if one enriches the basis of monomials with appropriate elements associated to the driving noise, then 
one can push the expansion further and make sense of the ``derivatives'' of the solutions, even 
in situations where they are \textit{not} 
differentiable in the classical sense. More precisely, if one introduces a symbol $\Xi$ for the noise and $\cI(\Xi)$ 
for the convolution of the noise with the heat kernel, then the solutions to a large class of nonlinear
variants of the stochastic heat equation admit a 
generalised Taylor expansion up to any arbitrary order $\gamma > 0$: this expansion will include a 
term proportional to 
$\cI(\Xi)$, encoding the fact that the solution will locally look like some multiple to the 
stochastic heat equation, as well as other ``non-standard'' higher-order terms. In general, the collection of 
monomials that should be added to the classical basis of Taylor monomials depends on the class of SPDEs at 
stake. In any case, this naturally leads to defining 
spaces $\cD^\gamma_{\infty,\infty}$ of coefficients on some enlarged basis, satisfying a very natural analogue to
the bounds to \eqref{Eq:BoundDg}.

While in the smooth case, the mapping from $\cD^\gamma_{\infty,\infty}$ into the space of H\"older functions was trivial (one has $F = f_0$), this is no longer the case in general, especially in situations where the
local expansion describes a distribution rather than a continuous function. In this case, one needs to 
show that there is a natural \textit{reconstruction operator} that associates a genuine function/distribution to any element in $\cD^\gamma_{\infty,\infty}$. In other words, one needs to show that the local descriptions prescribed by 
the generalised Taylor expansions can be patched together in a consistent way. Such a result is called a reconstruction theorem.

The upshot of this framework of generalised Taylor expansions is that it allows to make sense of some ill-posed products and opens the way to solving singular SPDEs. Indeed, this framework essentially reduces the problem to making sense of the products of basis elements, for example $\Xi$ times $\cI(\Xi)$. This is done through renormalisation procedures, we refer to~\cite{Hairer2014,Lorenzo,Ajay} for more details.

The paper is organised as follows. In Section \ref{Sec:Prelim}, we introduce the definitions necessary for our analysis and we define our spaces of modelled distributions. In Section \ref{Sec:Reconstruction}, we state and prove our reconstruction theorem. Section \ref{Sec:Embeddings} is devoted to the embedding theorems for modelled distributions. In the last section, we prove Schauder-type estimates at the level of our spaces of modelled distributions.

\subsection*{Acknowledgements}
MH gratefully acknowledges financial support from the
 Leverhulme Trust through a leadership award
 and from the European Research Council through a consolidator grant, project 615897.
CL acknowledges financial support from the ANR grant SINGULAR ANR-16-CE40-0020-01.

\section{Preliminaries}\label{Sec:Prelim}
\subsection{Spaces of distributions and wavelet analysis}\label{Section:Wavelet}

We are given a scaling $\s=(\s_1,\ldots,\s_d) \in \N^d$. Without further mention, we will always consider the $\s$-scaled ``norm'' $\|x\|_{\s} = \sup_{i=1,\ldots,d} |x_i|^{1/ \s_i}$ for all $x\in\R^d$, and $B(x,r)$ will denote the closed ball centred at $x$, and of radius $r$ with respect to the $\s$-scaled norm. Additionally, for any $k\in\N^d$ we will use the notation $X^k$ to denote the monomial $\prod_{i=1}^d X_i^{k_i}$, and we will call $|k|_{\s} = \sum_{i=1}^d \s_i k_i$ its \textit{scaled degree}.

Let $\cC^r(\R^d)$ be the space of functions $f:\R^d\rightarrow\R$ that admit continuous derivatives of order $k$, for all $k\in\N^d$ such that $|k|_\s \leq r$. We let $\BB^r(\R^d)$ be the subset of $\cC^\infty(\R^d)$ whose elements are supported in the centred ball of $\s$-scaled radius $1$ in $\R^d$ and are of $\cC^r$-norm bounded by $1$. Then, we let $\BB^r_\beta(\R^d)$ be the set of functions in $\BB^r(\R^d)$ that annihilate all polynomials of scaled degree at most $\beta$. We
also write $\BB^r_{-1}(\R^d)=\BB^r(\R^d)$.

From now on, $L^p$ will always refer to $L^p(\R^d,dx)$ and $x$ will be the associated integration variable, while $L^q_\lambda$ will be taken to be $L^q((0,1),\lambda^{-1}d\lambda)$ and $\lambda$ will be the associated integration variable. As usual, the notation $\langle f,g \rangle$ will 
be used both to denote the $L^2$-inner product of $f$ and $g$ and the evaluation of the 
distribution $f$ against the test function $g$.

We consider the space of \textit{tempered distributions} $\cD'(\R^d)$, that is, the topological dual of the Schwartz space $\cD(\R^d)$ of rapidly decreasing, infinitely differentiable functions.

\begin{definition}\label{Def:Besov}
Let $\alpha \in \R$, $p,q \in [1,\infty]$ and $r \in \N$ such that $r > |\alpha|$. For $\alpha < 0$, we let $\cB^{\alpha}_{p,q}(\R^d)$ be the space of distributions $\xi$ on $\R^d$ such that
\begin{equ}
\bigg\| \Big\| \sup_{\eta \in \BB^r(\R^d)}\frac{\big|\langle \xi, \eta_x^\lambda\rangle\big|}{\lambda^\alpha} \Big\|_{L^p} \bigg\|_{L^q_\lambda} < \infty\;.
\end{equ}
For $\alpha \geq 0$, this condition is replaced by
\begin{equ}\label{Eq:BoundDefBesovPos}
\Big\| \sup_{\eta \in \BB^r(\R^d)} \big|\langle \xi, \eta_x\rangle \big| \Big\|_{L^p} < \infty\;,\quad \bigg\| \Big\| \sup_{\eta \in \BB^r_{\lfloor \alpha \rfloor}(\R^d)}\frac{\big|\langle \xi, \eta_x^\lambda\rangle\big|}{\lambda^\alpha} \Big\|_{L^p} \bigg\|_{L^q_\lambda} < \infty\;.
\end{equ}
\end{definition}

Here, we used the notations $\eta_x$ and $\eta_x^\lambda$ as in \cite{Hairer2014}
to denote the test function $\eta$ recentred around $x$ and rescaled by $\lambda$.

\begin{remark}
For $\lambda\in (0,1]$, let $n\geq 0$ be the largest integer such that $2^{-n}\geq \lambda$. For any $\eta\in\BB^r$, the rescaled function $\eta^\lambda$  can always be viewed as some function $\psi^{2^{-n}}$ times a constant $C>0$, where $\psi\in\BB^r$. The constant $C$ is uniformly bounded over all $\lambda\in(0,1]$ and all $\eta\in\BB^r$. Consequently, the norm
$$\bigg\| \Big\| \sup_{\eta \in \BB^r(\R^d)}\frac{\big|\langle \xi, \eta_x^\lambda\rangle\big|}{\lambda^\alpha} \Big\|_{L^p} \bigg\|_{L^q_\lambda}\;,$$
can be replaced by
$$ \bigg(\sum_{n\ge 0} \Big\| \sup_{\eta \in \BB^r(\R^d)}\frac{\big|\langle \xi, \eta_x^\lambda\rangle\big|}{\lambda^\alpha} \Big\|_{L^p}^q \bigg)^{\frac1{q}}\;,$$
without altering the corresponding space of distributions, and similarly for \eqref{Eq:BoundDefBesovPos}.
\end{remark}

\begin{remark}
We choose to work with a scaling $\s$ in the Besov-space norm as this is better fitted for measuring regularity of solutions to PDEs where one direction - typically time - has a different scaling behaviour than the others - typically space. For instance, the parabolic operator $\partial_t - \partial^2_x$ is naturally associated with the scaling $\s =(2,1,\ldots,1)$ where the first coordinate is time and the $d-1$ others are space.
\end{remark}

There exists a simple characterisation of these spaces of distributions in terms of a \textit{wavelet analysis}; we refer to the works of Meyer~\cite{Meyer} and Daubechies~\cite{Ingrid} for more details on wavelet analysis, here we simply recall some basic facts. For every $r>0$, there exists a compactly supported function $\varphi\in\cC^r(\R)$ such that:\begin{enumerate}
\item We have $\langle\varphi(\cdot),\varphi(\cdot+k)\rangle=\delta_{k,0}$ for every $k\in\Z$,
\item There exist $a_k,k\in\Z$ with only finitely many non-zero values such that $\varphi(x)=\sum_{k\in\Z}a_k\varphi(2x-k)$ for every $x\in\R$,
\item For every polynomial $P$ of degree at most $r$, we have
\begin{equ}
\sum_{k\in\Z} \int_{y\in\R} P(y) \varphi(y-k) dy\, \varphi(x-k) = P(x)\;.
\end{equ}
\end{enumerate}
Given such a function $\varphi$, we set
\[ \varphi_x^n(y)=\prod_{i=1}^{d}2^{\frac{n\s_i}{2}}\varphi\big(2^{n\s_i}(y_i-x_i)\big) \;.\]
Then, we define an $\s$-scaled grid of mesh $2^{-n}$
\[ \Lambda_n:=\Big\{(2^{-n\s_1}k_1,\ldots,2^{-n\s_d}k_d):k_i\in\Z\Big\} \;,\]
and we let $V_n$ be the subspace of $L^2(\R^d)$ generated by $\{\varphi^{n}_x:x\in\Lambda_n\}$. Using the second defining property of the function $\varphi$, we deduce that $V_n \subset V_{n+1}$.

Finally, there exists a finite set $\Psi$ of compactly supported functions in $\cC^r$, that annihilate all polynomials of degree at most $r$, and such that for every $n\geq 0$,
\[ \{\varphi_x^n:x\in\Lambda_n\} \cup \{\psi_x^m:m\geq n, \psi\in\Psi, x\in\Lambda_m\} \;,\]
forms an orthonormal basis of $L^2(\R^d)$. Notice that the subspace of $L^2(\R^d)$ generated by the set $\{\psi_x^n:\psi\in\Psi, x\in\Lambda_n\}$ coincides with $V_n^\perp$, the orthogonal complement of $V_n$ into $V_{n+1}$. In the sequel, it will be convenient to denote by $\cP_n$ and $\cP_{n}^\perp$ the orthogonal projections on $V_n$ and $V_n^\perp$.

To simplify notation, we let $\ell^p_n$ be the Banach space of all sequences $u(x),x\in\Lambda_n$ such that
\begin{equ}
\big\| u(x) \big\|_{\ell^p_n} := \Big( \sum_{x\in\Lambda_n} 2^{-n|\s|} |u(x)|^p \Big)^{\frac{1}{p}} < \infty\;.
\end{equ}
We let appear $x$ in the norm in order to emphasise the associated integration variable: this will allow to distinguish parameters from integration variables later on. We also let $\ell^q$ be the usual Banach space of all sequences $u(n),n\in \N$ whose $\ell^q$-norm is finite. We will sometimes use the notation
\begin{equ}
	\|u(x)\|_{\ell^q(n\geq n_0)} := \Big(\sum_{n\geq n_0} u(n)^q\Big)^{\frac{1}{q}}\;,
\end{equ}
for any given $n_0\geq 1$. With all these definitions at hand, we have the following 
alternative characterisation of the Besov spaces $\cB^\alpha_{p,q}$.

\begin{proposition}\label{PropCharact}
Let $\alpha \in \R$ and $p,q \in [1,\infty]$. Take $r\in \N$ such that $r > |\alpha|$. Let $\xi$ be an element of $\cB^\alpha_{p,q}$, and set $a^{n,\psi}_x:=\langle \xi , \psi^n_x \rangle$, $x\in\Lambda_n$, $n\geq 0$, $\psi\in\Psi$ and $b^0_x:=\langle \xi , \varphi^0_x \rangle$, $x\in\Lambda_0$. Then, we have
\begin{equs}\label{Eq:CondWave}
\sup_{\psi\in\Psi}\bigg\| \Big\| \frac{a^{n,\psi}_x }{2^{-n\frac{|\s|}{2} -n \alpha}} \Big\|_{\ell^p_n} \bigg\|_{\ell^q} < \infty \;,\quad\Big\| b^0_x \Big\|_{\ell^p_0} &< \infty \;.
\end{equs}
Conversely, given two sequences $a^{n,\psi}_x$, $x\in\Lambda_n$, $n\geq 0$, $\psi\in\Psi$ and $b^0_x$, $x\in\Lambda_0$, such that (\ref{Eq:CondWave}) is satisfied, there exists a distribution $\xi \in \cB^\alpha_{p,q}$ whose evaluations against the wavelet basis are given by the coefficients $a^{n,\psi}_x$ and $b^0_x$.
\end{proposition}
As a consequence, (\ref{Eq:CondWave}) provides an equivalent norm to the $\cB^\alpha_{p,q}$ norm introduced in Definition \ref{Def:Besov}. In the sequel, the notation $\|\cdot\|_{\cB^\alpha_{p,q}}$ will refer indifferently to either of these two norms without further mention.
\begin{remark}\label{RkCharact}
These conditions can be restated at another scale. More precisely, given $n_0\geq 0$, the proposition still holds if (\ref{Eq:CondWave}) is replaced by
\begin{equ}[Eq:CondWave2]
\Big\|\frac{\langle \xi , \varphi^{n_0}_x \rangle}{2^{-n_0\frac{|\s|}{2}}} \Big\|_{\ell^p_{n_0}} < \infty \;,\quad \sup_{\psi\in\Psi}\bigg\| \Big\|\frac{\langle \xi , \psi^{n}_x \rangle }{2^{-n\frac{|\s|}{2} -n \alpha}}\Big\|_{\ell^p_n}\bigg\|_{\ell^q(n\geq n_0)} < \infty \;.\qquad
\end{equ}
\end{remark}
\begin{remark}
As it is pointed out in the introduction, this characterisation yields immediately the classical embedding theorems. In particular, we have the continuous inclusion $\cB^\alpha_{p,q} \subset \cB^\alpha_{p,\infty}$ that we will use at several occasions later on.
\end{remark}

This type of characterisation is classical, see for example \cite[Sec.~6.10]{Meyer}. 
The only specificity of the present result comes from the scaling $\s$ that we are working with.
\begin{proof}[of Proposition~\ref{PropCharact}]
Let $\xi\in\cB^\alpha_{p,q}$. There exists a constant $\kappa>0$, depending only on the size of the support of $\psi$ such that the following holds true. Uniformly over all $n\geq 1$, $\lambda\in [\kappa 2^{-n-1},\kappa 2^{-n})$, $x\in\Lambda_n$ and $y\in B(x,\lambda)$, the function $\psi^n_x$ is of the form $\eta^\lambda_y$ up to a constant multiplicative factor of order $2^{-n|\s|/2}$. Here, $\eta \in \BB^r_\beta$ with $\beta=-1$ when $\alpha \leq 0$, and $\beta=\lfloor \alpha \rfloor$ when $\alpha \geq 0$. Therefore, the definition of $\cB^\alpha_{p,q}$ ensures that the first condition of (\ref{Eq:CondWave}) holds. The second condition of (\ref{Eq:CondWave}) follows from similar arguments.

Conversely, we assume that (\ref{Eq:CondWave}) holds. We need to show that for all $\eta\in\BB^r$
\begin{equ}\label{Eq:Defxi}
\sum_{y\in\Lambda_0} b^0_y \langle \varphi_y^0 , \eta^\lambda_x \rangle + \sum_{\psi\in\Psi} \sum_{n\geq 0} \sum_{y\in\Lambda_n} a^{n,\psi}_y \langle \psi_y^n , \eta^\lambda_x \rangle \;,
\end{equ}
converges and satisfies the bound(s) of Definition \ref{Def:Besov}. Once this is established, we simply define $\langle \xi,\eta^\lambda_x\rangle$ as the value of this series. Then, it is elementary to check that this can be extended into a genuine distribution that belongs to $\cB^\alpha_{p,q}$.

Let $M$ be the maximum of the sizes of the support (for the scaled distance) of $\varphi$ and $\psi\in\Psi$. We start with the first term of (\ref{Eq:Defxi}). Set $\beta=\lfloor \alpha \rfloor+1$ if $\alpha \geq 0$, otherwise set $\beta = 0$. Using the Taylor expansion of $\varphi_y^0$ at $x$, we deduce that $|\langle \varphi_y^0 , \eta^\lambda_x \rangle| \lesssim \lambda^{\beta}$ uniformly over all $x,y \in \R^d$, all $\lambda\in (0,1]$ and all $\eta\in\BB^r_{\beta-1}$. Furthermore, this inner product vanishes as soon as $\|x-y\|_{\s} > \lambda + M$, so that there are only finitely many $y\in\Lambda_0$ with a non-zero contribution, uniformly over all $\lambda\in (0,1]$ and $x\in\R^d$. For all $\alpha \in \R$, we get
\begin{equs}
\Big\| \sup_{\eta\in\BB^r}\Big| \sum_{y\in\Lambda_0} b^0_y \langle \varphi_y^0 , \eta_x \rangle \Big| \Big\|_{L^p} &\lesssim \Big(\int_{x\in\R^d} \sum_{y\in\Lambda_0:\|x-y\|_{\s} \leq \lambda + M} |b^0_y|^p dx \Big)^{\frac1{p}}\\
&\lesssim \Big(\sum_{y\in\Lambda_0} \big|b^0_y \big|^p \Big)^{\frac{1}{p}}\;,
\end{equs}
which is finite by (\ref{Eq:CondWave}). Similarly, for $\alpha \geq 0$ we get
\begin{equs}
\bigg\| \Big\| \sup_{\eta\in\BB^r_{\beta-1}}\frac{\big| \sum_{y\in\Lambda_0} b^0_y \langle \varphi_y^0 , \eta^\lambda_x \rangle \big|}{\lambda^\alpha} \Big\|_{L^p} \bigg\|_{L^q}& \lesssim \big\|\lambda^{\beta - \alpha} \big\|_{L^q} \Big(\sum_{x\in\Lambda_0} \big|b^0_x \big|^p \Big)^{\frac{1}{p}}\;,
\end{equs}
which is finite by (\ref{Eq:CondWave}) and since $\beta-\alpha > 0$.

We now turn to the second term of (\ref{Eq:Defxi}). Fix $\psi\in\Psi$. For any $\lambda\in (0,1]$, we let $n_0$ be the largest integer such that $2^{-n_0} \geq \lambda$. We need to argue differently according to the relative values of $n$ and $n_0$.

We start with the case $n<n_0$, which does not cover the first bound of (\ref{Eq:BoundDefBesovPos}). Take $\beta$ as above. Using the Taylor expansion of $\psi$ at $x$, we deduce that $|\langle \psi_y^n , \eta^\lambda_x \rangle| \lesssim 2^{n(\frac{|\s|}{2} + \beta)}\lambda^\beta$ uniformly over all $x,y \in \R^d$, all $\lambda\in (0,1]$, all $\eta\in\BB^r_{\beta-1}$ and all $n<n_0$. Furthermore, this inner product vanishes as soon as $\|x-y\|_{\s} > \lambda + M 2^{-n}$ so that only finitely many $y\in\Lambda_n$ yield a non-zero contribution, uniformly over all the parameters. Using the triangle inequality at the second line and Jensen's inequality on the sum over $y$ at the fourth line, we get
\begin{equs}
{}&\bigg\| \Big\| \sup_{\eta\in\BB^r_{\beta-1}} \frac{\big| \sum_{n<n_0} \sum_{y\in\Lambda_n} a^{n,\psi}_y \langle \psi_y^n , \eta^\lambda_x \rangle \big|}{\lambda^\alpha} \Big\|_{L^p} \bigg\|_{L^q}\\
&\lesssim\bigg\| \sum_{n<n_0} \Big\| \sum_{y\in\Lambda_n} \sup_{\eta\in\BB^r_{\beta-1}} \frac{\big|   a^{n,\psi}_y \langle \psi_y^n , \eta^\lambda_x \rangle \big|}{\lambda^\alpha} \Big\|_{L^p} \bigg\|_{L^q}\\
&\lesssim \bigg\| \sum_{n<n_0} \bigg( \int_{x\in\R^d} \bigg( \sum_{y\in\Lambda_n: \|x-y\|_{\s} \leq \lambda + M 2^{-n}} \frac{\big|a^{n,\psi}_y \big|}{\lambda^\alpha} 2^{n(\frac{|\s|}{2}+\beta)} \lambda^{\beta} \bigg)^p dx \bigg)^{\frac1{p}}\bigg\|_{L^q}\\
&\lesssim \bigg\| \sum_{n<n_0} 2^{n(\beta-\alpha)}\lambda^{\beta-\alpha} \Big( \sum_{y\in\Lambda_n} 2^{-n|\s|} \Big(\frac{\big|a^{n,\psi}_y \big|}{2^{-n(\frac{|\s|}{2}+\alpha)}}\Big)^p \Big)^{\frac{1}{p}} \bigg\|_{L^q}\;.
\end{equs}
At this point, we observe that $\sum_{n<n_0} 2^{n(\beta-\alpha)}\lambda^{\beta-\alpha}$ is of order $1$, uniformly over all $\lambda\in (0,1]$. Consequently, Jensen's inequality and a simple integration over $\lambda$ ensure that the last expression is bounded by a term of order
\begin{equ}
\bigg\|\Big\| \frac{\big|a^{n,\psi}_x \big|}{2^{-n(\frac{|\s|}{2}+\alpha)}}\Big\|_{\ell^p_n} \bigg\|_{\ell^q}\;,
\end{equ}
which is finite by (\ref{Eq:CondWave}).

We consider the case where $n\geq n_0$. Using the Taylor expansion of $\eta$ at $y$, we deduce that $|\langle \psi_y^n , \eta^\lambda_x \rangle| \lesssim 2^{-n(\frac{|\s|}{2} + r)}\lambda^{-|\s|-r}$ uniformly over all $x,y \in \R^d$, all $\lambda\in (0,1]$, all $\eta\in\BB^r$ and all $n\geq n_0$. Furthermore, the inner product vanishes as soon as $\|x-y\|_{\s} > \lambda + M 2^{-n}$ so that, for any given $x\in\R^d$, there are of the order of $2^{(n-n_0)|\s|}$ terms with a non-zero contribution in the sum over $y\in\Lambda_n$, uniformly over all the parameters. We first assume that $\alpha \geq 0$ and take $\lambda = 1$ (so $n_0=0$), in order to obtain the first bound of (\ref{Eq:BoundDefBesovPos}). Using the triangle inequality on the sum over $n$ at the first line, Jensen's inequality on the sum over $y$ at the second line, and the H\"older inequality at the third line, we get
\begin{equs}
\Big\| \sup_{\eta\in\BB^r} \Big| \sum_{n\geq 0} \sum_{y\in\Lambda_n} a^{n,\psi}_y \langle \psi_y^n&, \eta_x \rangle \Big| \Big\|_{L^p} \lesssim \sum_{n\geq 0} \Big\| \!\!\!\!\!\sum_{\substack{y\in\Lambda_n\\|x-y|_{\s} \leq 1 + M 2^{-n}}}\!\!\!\!\! |a^{n,\psi}_y| 2^{-n(\frac{|\s|}{2} + r)} \Big\|_{L^p}\\
&\lesssim \sum_{n\geq 0} 2^{-n(r+\alpha)} \Big( \sum_{y\in\Lambda_n} 2^{-n|\s|} \Big(\frac{\big|a^{n,\psi}_y \big|}{2^{-n(\frac{|\s|}{2}+\alpha)}}\Big)^p \Big)^{\frac{1}{p}}\\
&\lesssim \bigg\| \Big\| \frac{\big|a^{n,\psi}_x \big|}{2^{-n(\frac{|\s|}{2}+\alpha)}}\Big\|_{\ell^p_n} \bigg\|_{\ell^q} \Big(\sum_{n\geq 0} 2^{-n\frac{q}{q-1}(r+\alpha)}\Big)^{1-\frac{1}{q}}\;,
\end{equs}
where $\frac{q}{q-1}$ is set to $+\infty$ when $q=1$, and to $1$ when $q=+\infty$. Since $r+\alpha > 0$, this is finite.

We now consider any $\alpha$ and do no longer impose $\lambda =1$. Using the triangle inequality at the second line, and Jensen's inequality at the third line, we get
\begin{equs}
{}&\bigg\| \Big\| \sup_{\eta\in\BB^r} \frac{\big| \sum_{n\geq n_0} \sum_{y\in\Lambda_n} a^{n,\psi}_y \langle \psi_y^n , \eta^\lambda_x \rangle \big|}{\lambda^\alpha} \Big\|_{L^p} \bigg\|_{L^q}\\
&\qquad\lesssim \bigg\| \sum_{n\geq n_0} \Big\| \sum_{y\in\Lambda_n: \|x-y\|_{\s} \leq \lambda + M 2^{-n}} \frac{\big|a^{n,\psi}_y \big|}{\lambda^\alpha} 2^{-n(\frac{|\s|}{2} + r)}\lambda^{-|\s|-r} \Big\|_{L^p}\bigg\|_{L^q}\\
&\qquad\lesssim \bigg\| \sum_{n\geq n_0} 2^{-(n-n_0)(r+\alpha)} \Big( \sum_{y\in\Lambda_n} 2^{-n|\s|} \Big(\frac{\big|a^{n,\psi}_y \big|}{2^{-n(\frac{|\s|}{2}+\alpha)}}\Big)^p \Big)^{\frac{1}{p}} \bigg\|_{L^q}\;.
\end{equs}
Since $\sum_{n\geq n_0} 2^{-(n-n_0)(r+\alpha)}$ is of order $1$, Jensen's inequality ensures that the last expression is bounded by a term of order
\begin{equ}
 \bigg(\sum_{n\geq 0} 2^{-n(r+\alpha)}\int_{\lambda \in (2^{-n},1]}\frac{d\lambda}{\lambda^{r+\alpha+1}}  \Big\|\frac{\big|a^{n,\psi}_x\big|}{2^{-n(\frac{|\s|}{2}+\alpha)}}\Big\|_{\ell^p_n}^q \bigg)^{\frac{1}{q}}
\lesssim \bigg\|  \Big\|\frac{\big|a^{n,\psi}_x \big|}{2^{-n(\frac{|\s|}{2}+\alpha)}} \Big\|_{\ell^p_n}\bigg\|_{\ell^q}\;,
\end{equ}
which is finite by (\ref{Eq:CondWave}). This concludes the proof.
\end{proof}

We conclude this subsection with an elementary property.

\begin{lemma}\label{Lemma:BesovFunction}
Let $\xi \in \cB^\alpha_{p,q}$ for some $\alpha > 0$. Let $\rho:\R^d\rightarrow\R$ be a smooth, even function supported in $B(0,1)$ and integrating to $1$. Then $\big\{x \mapsto \langle \xi,\rho^\lambda_x\rangle, x\in\R^d\}_{\lambda\in (0,1]}$, is a Cauchy family in $L^p$ for $\lambda \downarrow 0$ and its limit coincides with the distribution $\xi$.
\end{lemma}
\begin{proof}
Set $\tilde{\xi}^\lambda(x) := \langle \xi,\rho^\lambda_x\rangle$. For all $\lambda > \lambda' \in (0,1]$, we can write
\begin{equ}\label{Eq:lambda}
\rho^\lambda_x-\rho^{\lambda'}_x = \big(\rho^\lambda_x-\rho^{2^{-n_0}}_x\big) + \sum_{n=n_0}^{n_1-1} \big(\rho^{2^{-n}}_x-\rho^{2^{-(n+1)}}_x\big) + \big(\rho^{2^{-n_1}}_x-\rho^{\lambda'}_x\big)\;,
\end{equ}
where $n_0, n_1$ are the largest integers such that $2^{-n_0}\geq \lambda$ and $2^{-n_1}\geq \lambda'$. By the classical embeddings, $\xi$ belongs to $\cB^\epsilon_{p,\infty}$ for any $\epsilon \in (0,1\wedge \alpha)$. Since every term inside brackets is a smooth function integrating to $0$, it is simple to check using (\ref{Eq:lambda}) that the family $\{\tilde{\xi}^\lambda\}_{\lambda\in (0,1]}$ is Cauchy in $L^p$. Let $\tilde{\xi}$ be its limit: it naturally defines a distribution on $\R^d$. Let $\eta$ be a compactly supported, smooth function on $\R^d$. We set
\begin{equ}
\eta_\lambda(y) = \int \rho^\lambda(y-x)\eta(x)dx\;.
\end{equ}
Observe that $\eta_\lambda$ is supported in a compact set whose diameter is of order $1$, uniformly over all $\lambda\in(0,1]$. We write
\begin{equ}\label{Eq:Decompxitilde}
\langle \tilde{\xi}-\xi , \eta_\lambda\rangle = \langle \tilde{\xi} - \tilde{\xi}^\lambda  , \eta_\lambda\rangle + \langle \tilde{\xi}^\lambda - \xi,\eta_\lambda\rangle\;,
\end{equ}
and we show that each term on the right vanishes as $\lambda\downarrow 0$. Indeed, using H\"older's inequality and the fact that $\eta^\lambda$ is a compactly supported, smooth function, we find
\begin{equ}
\big| \langle \tilde{\xi} - \tilde{\xi}^\lambda , \eta_\lambda \rangle \big| \lesssim \int_y |\eta_\lambda(y)| \big|\tilde{\xi}(y) - \tilde{\xi}^\lambda(y) \big| dy
\lesssim \Big\| \tilde{\xi} - \tilde{\xi}^\lambda \Big\|_{L^p} \;,
\end{equ}
so that it vanishes as $\lambda\downarrow 0$. On the other hand, we can rewrite the second term on the right hand side of (\ref{Eq:Decompxitilde}) as follows:
\begin{equ}
\langle \tilde{\xi}^\lambda - \xi,\eta_\lambda\rangle = \int_z \eta(x) \int_y \rho^\lambda_x(y) \langle \xi, \rho^\lambda_y - \rho^\lambda_x\rangle dy\,dx\;.
\end{equ}
Since $\rho^\lambda_y - \rho^\lambda_x$ integrates to $0$ and since $\xi$ belongs to $\cB^\epsilon_{p,\infty}$ for some $\epsilon >0$, we conclude that this last term vanishes as $\lambda\downarrow 0$. We have proven that $\langle \tilde{\xi}-\xi , \eta_\lambda\rangle$ goes to $0$ as $\lambda\downarrow 0$. Since $\eta_\lambda$ converges to $\eta$ in the topology on $\cD(\R^d)$ (that we introduced at the beginning of the section), we deduce that $\langle \tilde{\xi}-\xi , \eta\rangle=0$ thus concluding the proof.
\end{proof}

\subsection{Regularity structures}

Recall that a regularity structure is a triple $(\cA,\cT,\cG)$ where:\begin{enumerate}
\item $\cA$, the set of \textit{homogeneities}, is a subset of $\R$ assumed to be locally finite and bounded from below,
\item $\cT$, the \textit{model space}, is a graded vector space $\bigoplus_{\zeta\in\cA} \cT_\zeta$ consisting of finite sequences indexed by $\cA$, and each $\cT_\zeta$ is a Banach space,
\item $\cG$, the \textit{structure group}, is a group of continuous linear transformations on $\cT$ such that for every $\Gamma\in\cG$, every $\zeta \in\cA$ and every $\tau\in\cT_\zeta$, we have $\Gamma \tau - \tau \in \cT_{<\zeta}$ where $\cT_{<\zeta} = \bigoplus_{\beta < \zeta} \cT_\beta$.
\end{enumerate}
An elementary example of regularity structures is the \textit{polynomial} regularity structure $(\bar\cA,\bar\cT,\bar\cG)$ defined as follows. Take $\bar\cA = \N$, and for every $\zeta\in\N$, let $\cT_\zeta$ be the set of all polynomials in $X_i, i=1 \ldots d$ with $\s$-scaled degree equal to $\zeta$. Recall that the $\s$-scaled degree of $X^k=\prod_{i=1}^d X_i^{k_i}$ is given by $|k|_\s = \sum_i \s_i k_i$, for any $k\in\N^d$. Furthermore, the structure group $\bar\cG$ is taken to be the group of translations on $\R^d$ acting on polynomials in the usual way.

We will denote by $\cQ_\zeta$ or $(\cdot)_\zeta$ the projection from $\cT$ onto $\cT_\zeta$, and $|\tau|_\zeta$ will denote the norm of the projection of $\tau$ onto $\cT_\zeta$ for all $\tau\in\cT$.
Given a regularity structure $(\cA,\cT,\cG)$, recall the notion of \textit{model} that endows every element in the structure with some analytical features. From now on, we let $r\in\N$ be such that $r> |\zeta|$ for all $\zeta\in\cA_{\gamma}:=\cA\cap (-\infty,\gamma)$ for some fixed $\gamma > 0$.

\begin{definition}
A model is a pair $(\Pi,\Gamma)$ that satisfies the following conditions. First of all, $\Pi =(\Pi_x)_{x\in\R^d}$ is a collection, indexed by $\R^d$, of linear maps from $\cT_{<\gamma}$ into the set of distributions $\cD'(\R^d)$ such that
\begin{equ}
\| \Pi\|_x = \sup_{\eta \in\BB^r} \sup_{\lambda\in (0,1]} \sup_{\zeta \in\cA_{\gamma}} \sup_{\tau \in \cT_\zeta} \frac{\big|\langle \Pi_x \tau , \eta^\lambda_x\rangle \big|}{|\tau| \lambda^\zeta} \lesssim 1\;,
\end{equ}
uniformly over all $x\in\R^d$. Second $\Gamma = (\Gamma_{x,y})_{x,y\in\R^d}$ where every $\Gamma_{x,y}$ belongs to the structure group $\cG$ and we have
\begin{equ}
\| \Gamma \|_{x,y} = \sup_{\beta \leq \zeta \in \cA_{\gamma}} \sup_{\tau\in \cT_\zeta} \frac{\big| \Gamma_{x,y} \tau \big|_{\beta}}{|\tau| \|x-y\|_\s^{\zeta-\beta}} \lesssim 1\;,
\end{equ}
uniformly over all $x\in\R^d$ and all $y\in B(x,1)$. We also set $\|\Pi\|:=\sup_x \| \Pi\|_x$ and $\|\Gamma\| := \sup_{x,y} \| \Gamma \|_{x,y}$.
\end{definition}

\begin{remark}
Unlike in \cite{Hairer2014}, we assume here that the bounds on $\Pi$ and $\Gamma$ hold uniformly over $x \in \R^d$.
This is required since Besov spaces measure not only the local properties of a function but also its global integrability.
It would of course be possible to adapt the results of this article to build analogues of weighted or 
local Besov spaces in which some non-uniformity in these bounds is allowed, we refer to~\cite{DavidJosef} for instance.
\end{remark}

\subsection{Modelled distributions}

Given a regularity structure $(\cA,\cT,\cG)$, and a model $(\Pi,\Gamma)$, we introduce some spaces of \textit{modelled distributions} that mimic the spaces $\cB^\alpha_{p,q}$ in the framework of regularity structures. Recall the notation $L^p$ introduced earlier in the paper.
We also henceforth write $\cA_\gamma = \cA \cap (-\infty,\gamma)$.

\begin{definition}\label{Def:Dgamma}
For $\gamma \in \R$, let $\cD^\gamma_{p,q}$ be the Banach space of all measurable maps $f: \R^d \rightarrow \cT_{<\gamma}$ such that, for all $\zeta\in\cA_\gamma$, we have:\begin{enumerate}
\item Local bound:
\begin{equ}
\Big\| \big| f(x) \big|_\zeta \Big\|_{L^p} < \infty\;,
\end{equ}
\item Translation bound:
\begin{equ}
\bigg(\int_{h\in B(0,1)}\bigg\| \frac{\big| f(x+h)-\Gamma_{x+h,x} f(x) \big|_\zeta}{\|h\|_\s^{\gamma-\zeta}} \bigg\|^q_{L^p} \frac{dh}{\|h\|_\s^{|\s|}}\bigg)^{\frac1{q}} < \infty\;.
\end{equ}
\end{enumerate}
We write $\$f\$$ for the corresponding norm.
\end{definition}

This definition is close to the definition of classical Besov spaces via differences, 
see for instance \cite[Sec.~2.5.12]{Triebel}.
Note also that in the particular case $q = p$, this definition coincides with the definition
of the spaces $\CD^\gamma_p(\R^d)$ given in \cite{DavidJosef}.
The main trick for proving the embedding theorems for the spaces $\cD^\gamma_{p,q}$ is to work at the level of averages over balls of radius $2^{-n}$. We define $\cE_n := B(0,2^{-n}) \cap \Lambda_n \backslash\{0\}$, that is
\begin{equ}
\cE_n=\big\{h\in\R^d: 2^{n\s_i} h_i \in \{-1,0,1\}\big\} \backslash \{0\}\;.
\end{equ}

\begin{definition}\label{Def:Dgammabar}
For $\gamma \in \R$, let $\bar{\cD}^\gamma_{p,q}$ be the Banach space of all sequences, indexed by $n\geq 0$, of maps $\aver{f}{n}: \Lambda_n \rightarrow \cT_{<\gamma}$ such that for all $\zeta\in\cA_\gamma$, we have:\begin{enumerate}
\item Local bound:
\begin{equ}
\Big\| \big| \aver{f}{0} (x) \big|_\zeta \Big\|_{\ell^p_0} < \infty\;,
\end{equ}
\item Translation bound:
\begin{equ}
\bigg(\sum_{n\geq 0} \sum_{h\in \cE_n} \bigg\| \frac{\big| \aver{f}{n}(x+h)-\Gamma_{x+h,x} \aver{f}{n}(x) \big|_\zeta}{2^{-n(\gamma-\zeta)}} \bigg\|^q_{\ell^p_n} \bigg)^{\frac1{q}} < \infty\;,
\end{equ}
\item Consistency bound:
\begin{equ}
\bigg(\sum_{n\geq 0} \bigg\| \frac{\big|\aver{f}{n} (x) - \aver{f}{n+1}(x)\big|_\zeta}{2^{-n(\gamma-\zeta)}}\bigg\|_{\ell^p_n}^q \bigg)^{\frac1{q}} < \infty\;.
\end{equ}
\end{enumerate}
\end{definition}
As for the $\cD^\gamma$-norm, we use the notation $\$\bar{f}\$$ for the $\bar{\cD}^\gamma$-norm: this will never raise any confusion in the sequel.

\begin{remark}\label{Rmk:Consistency}
Let $\cE^C_{n} = B(0,C2^{-n}) \cap \Lambda_n \backslash\{0\}$ for some constant $C>0$. Combining the translation and consistency bounds, we get
\begin{equ}
\bigg(\sum_{n\geq 0} \sum_{h\in \cE^C_{n+1}} \bigg\| \frac{\big|\aver{f}{n} (x) - \Gamma_{x,x+h} \aver{f}{n+1}(x+h)\big|_\zeta}{2^{-n(\gamma-\zeta)}}\bigg\|_{\ell^p_n} \bigg)^{\frac1{q}} \lesssim \$\bar{f}\$\;.
\end{equ}
\end{remark}

\begin{notation}
We will write $f_\zeta(x)$ and $\aver{f}{n}_\zeta(x)$ as shortcuts for $\cQ_\zeta f(x)$ and $\cQ_\zeta \aver{f}{n}(x)$. We will also write $|f(x)|_\zeta$ and $|\aver{f}{n}(x)|_\zeta$ for $|f_\zeta(x)|$ and $|\aver{f}{n}_\zeta(x)|$.
\end{notation}

One should think of $\aver{f}{n}(x)$ as being a suitable average of some function $f$ over a ball 
of radius $2^{-n}$ centred at $x$. This will be made more precise in Theorem~\ref{Th:EqDgamma} below.
We first show that, although the local bound is imposed for averages over balls of 
radius $1$ only, the consistency and translation bounds allow one to propagate this 
bound to averages over balls of arbitrarily small radius.

\begin{lemma}\label{Lemma:LocalBoundn}
Let $\bar{f}\in\bar{\cD}^{\gamma}_{p,q}$. Then for all $\zeta \in \cA_\gamma$, we have
\begin{equ}
\sup_{n\geq 0} \Big\| \big| \aver{f}{n} (x) \big|_\zeta \Big\|_{\ell^p_n} < \infty\;.
\end{equ}
\end{lemma}

\begin{proof}
It is sufficient to prove the bound of the lemma with $\zeta$ taken to be the largest element in $\cA_\gamma$. Indeed, if we let $\beta$ be the second largest element in $\cA_\gamma$, then the bound on $\bar{f}_\zeta$ easily implies that the restriction of $\bar f$ to $\cT_{<\tilde{\gamma}}$ belongs to $\cD^{\tilde{\gamma}}_{p,q}$ for all $\tilde{\gamma} \in (\beta,\zeta)$ (one can take $\tilde{\gamma} = \zeta$ in the case where $q=\infty$). Consequently, the bound holds true also for $\beta$, and by recursion, for all levels in $\cA_\gamma$. 
We are left with the proof of the bound for $\zeta = \max \cA_\gamma$. The key argument is the following decomposition. For any $y\in\Lambda_{n+1}$, let $x_y :=\sup\{x\in\Lambda_n: x_i \leq y_i\;\forall i\}$ 
(here, the supremum refers to the lexicographic order) and write
\begin{equ}
\aver{f}{n+1}_\zeta(y) = \aver{f}{n}_\zeta(x_y) + \aver{f}{n+1}_\zeta(y)-\aver{f}{n}_\zeta(x_y)\;.
\end{equ}
By the triangle inequality, we have the bound $\|\aver{f}{n+1}_\zeta\|_{\ell^{p}_{n+1}} \leq A_1(n) + A_2(n)$ where
\begin{equs}
A_1(n) &= \Big(\sum_{y\in\Lambda_{n+1}} 2^{-(n+1)|\s|} \big|\aver{f}{n}_\zeta(x_y)\big|^{p} \Big)^{\frac1{p}}\;,\\
A_2(n) &= \bigg(\sum_{y\in\Lambda_{n+1}} 2^{-(n+1)|\s|} \Big| \aver{f}{n+1}_\zeta(y)-\aver{f}{n}_\zeta(x_y)\Big|^{p} \bigg)^{\frac1{p}}\;.
\end{equs}
We bound separately these two terms. A simple combinatorial argument ensures that
\begin{equ}
A_1(n) = \Big(\sum_{x\in\Lambda_{n}} 2^{-n|\s|}\big|\aver{f}{n}_\zeta(x)\big|^{p} \Big)^{\frac1{p}} = \big\|\aver{f}{n}_\zeta\big\|_{\ell^{p}_{n}}\;.
\end{equ}
We turn to $A_2$. There exists $C>0$ (independent of $f$) such that
\begin{equs}
A_2(n) &\lesssim \bigg(\sum_{x\in\Lambda_{n}} 2^{-n|\s|} \sum_{h\in\cE_{n+1}^C} \Big|\aver{f}{n+1}_\zeta(x+h)-\aver{f}{n}_\zeta(x)\Big|^{p} \bigg)^{\frac1{p}}\\
&\lesssim \sum_{h\in\cE_{n+1}^C}\bigg(\sum_{x\in\Lambda_{n}} 2^{-n|\s|} \Big|\aver{f}{n+1}_\zeta(x+h)-\aver{f}{n}_\zeta(x)\Big|^{p} \bigg)^{\frac1{p}}\;,
\end{equs}
uniformly over all $n\geq 0$. Since $\zeta = \max \cA_\gamma$, we have the identity $\cQ_\zeta \Gamma_{x,y} \tau = \cQ_\zeta\tau$ for all $\tau \in \cT_{<\gamma}$. By H\"older's inequality, this yields
\begin{equs}
\sum_{n\geq 0} A_2(n) &\lesssim \bigg(\sum_{n\geq 0} \sum_{h\in \cE_{n+1}^C} \bigg\| \frac{\big|\aver{f}{n} (x) - \Gamma_{x,x+h} \aver{f}{n+1}(x+h)\big|_\zeta}{2^{-n(\gamma-\zeta)}}\bigg\|_{\ell^p_n}^q \bigg)^{\frac1{q}}\\
&\quad\times\Big(\sum_{n\geq 0} 2^{-n(\gamma-\zeta)\bar q}\Big)^{\frac1{\bar q}}\;,
\end{equs}
where $\bar q \in [1,\infty]$ is the conjugate of $q$. Combining these bounds, we deduce that there exists $K>0$ such that
\begin{equ}\label{Eq:LocalBdn}
\|\aver{f}{n_0}_\zeta\|_{\ell^{p}_{n_0}} \leq \|\aver{f}{0}_\zeta\|_{\ell^{p}_{0}} + K \bigg(\sum_{n\geq 0} \sum_{h\in \cE_{n+1}^C} \bigg\| \frac{\big|\aver{f}{n} (x) -\Gamma_{x,x+h}\aver{f}{n+1}(x+h)\big|_\zeta}{2^{-n(\gamma-\zeta)}}\bigg\|_{\ell^p_n}^q \bigg)^{\frac1{q}}\;,
\end{equ}
uniformly over all $n_0\geq 0$. By Remark \ref{Rmk:Consistency}, this concludes the proof.
\end{proof}

The following result shows that the spaces $\cD$ and $\bar{\cD}$ are essentially equivalent. Recall that the notation $B(x,r)$ refers to the ball in $\R^d$, centred at $x$ and of radius $r$ for the scaled distance.

\begin{theorem}\label{Th:EqDgamma}
Let $f\in\cD^\gamma_{p,q}$, and set for all $n\geq 0$ and all $x\in\Lambda_n$
\begin{equ}\label{Eq:ffbar}
\aver{f}{n}(x) = \int_{B(x,2^{-n})} 2^{n|\s|} \Gamma_{x,y} f(y) dy\;.
\end{equ}
Then $\bar{f}$ belongs to $\bar{\cD}^{\gamma}_{p,q}$.

Conversely, let $\bar{f}\in\bar{\cD}^{\gamma}_{p,q}$ and for all $n\geq 0$ and all $x\in\R^d$, define $f_n(x) = \Gamma_{x,x_n} \aver{f}{n}(x_n)$ where $x_n$ is the nearest point to $x$ in $\Lambda_n$. 
Then,  the sequence $(f_n)_{n\geq 0}$ converges in $L^p$ to some limit $f \in \cD^{\gamma}_{p,q}$.

In the case where $\bar{f}$ is built from some $f\in\cD^\gamma_{p,q}$ as in the first part of the statement, then the element built in the second part of the statement coincides with $f$.
\end{theorem}

Let us observe that that our map $f\mapsto \bar{f}$ is far from being canonical: that is, 
one could opt for slightly different ways of performing the average, leading to an 
alternative definition of this map, but without altering the statement of the theorem.

\begin{proof}
The first part of the statement is elementary to prove, so we leave the details to the interested reader. Let us turn to the converse statement. Take $\bar{f}\in\bar{\cD}^{\gamma}_{p,q}$, define $f_n$ as in the statement, and fix some $\zeta \in \cA_\gamma$. Recall that by the definition of a model one
has $|\Gamma_{x,x+h}\tau|_\zeta \lesssim |\tau| \|h\|_\s^{\beta-\zeta}$ for all $\tau\in\cT_{\beta}$ and all $\zeta\leq \beta$. We deduce the bound
\begin{equs}[Eq:BoundIncrLp]
{}&\Big\| \big| f_n(x)-f_{n+1}(x) \big|_\zeta \Big\|_{L^p}\\
&\leq\Big\| \big| \Gamma_{x,x_n}\big(\aver{f}{n}(x_n)-\Gamma_{x_n,x_{n+1}}\aver{f}{n+1}(x_{n+1})\big) \big|_\zeta \Big\|_{L^p}\\
&\leq \sum_{\beta \geq \zeta} 2^{-n(\beta-\zeta)}\Big\| \big| \aver{f}{n}(x_n)-\Gamma_{x_n,x_{n+1}}\aver{f}{n+1}(x_{n+1}) \big|_\beta \Big\|_{L^p}\\
&\leq \sum_{\beta \geq \zeta}\sum_{h\in\cE_{n+1}^C} 2^{-n(\beta-\zeta)}\Big\| \big| \aver{f}{n}(x_n)-\Gamma_{x_n,x_n+h}\aver{f}{n+1}(x_n+h) \big|_\beta \Big\|_{L^p}\\
&\lesssim \sum_{\beta \geq \zeta}\sum_{h\in\cE_{n+1}^C} 2^{-n(\beta-\zeta)}\Big\| \big| \aver{f}{n}(x)-\Gamma_{x,x+h}\aver{f}{n+1}(x+h) \big|_\beta \Big\|_{\ell^p_n}\;,
\end{equs}
uniformly over all $n\geq 0$. Let $\bar q$ be the conjugate of $q$. By H\"older's inequality and (\ref{Eq:BoundIncrLp}), we get
\begin{equs}
{}&\sum_{n\geq n_0}\Big\| \big| f_n(x)-f_{n+1}(x) \big|_\zeta \Big\|_{L^p}\\
&\leq \bigg(\sum_{n\geq n_0}\Big\| \frac{\big| f_n(x)-f_{n+1}(x) \big|_\zeta }{2^{-n(\gamma-\zeta)}}\Big\|_{L^p}^q\bigg)^{\frac1{q}} \Big(\sum_{n\geq n_0} 2^{-n(\gamma-\zeta)\bar q}\Big)^{\frac1{\bar q}}\\
&\lesssim \$\bar f\$ 2^{-n_0(\gamma-\zeta)}\;,
\end{equs}
uniformly over all $n_0\geq 0$. This shows that $\cQ_\zeta f_n$ is a Cauchy sequence in $L^p$.
Since this is true for every $\zeta$, it follows that $f_n$ is Cauchy in $L^p$ and we write 
$f$ for its limit. We need to show that this defines an element of $\cD^\gamma_{p,q}$. The local bound is already proved by construction, so we focus on the translation bound. For any $h\in B(0,1)$, let $n_0$ be the largest integer such that $2^{-n_0} \geq \|h\|_\s$. We write
\begin{equs}\label{Eq:DecompTranslBound}
f(x+h) &- \Gamma_{x+h,x}f(x) = \big(f(x+h) - f_{n_0}(x+h)\big)\\
&\quad + \big(f_{n_0}(x+h) - \Gamma_{x+h,x}f_{n_0}(x)\big)+ \Gamma_{x+h,x}\big(f_{n_0}(x) - f(x)\big)\;,
\end{equs}
and we bound  these three terms separately. First of all, we observe that
\begin{equ}
f_{n_0}(x+h) - \Gamma_{x+h,x}f_{n_0}(x) = \Gamma_{x+h,(x+h)_{n_0}}\big(\aver{f}{n_0}((x+h)_{n_0}) - \Gamma_{(x+h)_{n_0},x_{n_0}} \aver{f}{n_0}(x_{n_0})\big)\;.
\end{equ}
Therefore, if we define the annulus 
$A(0,n_0):= B(0,2^{-n_0})\backslash B(0,2^{-n_0-1})$, we easily deduce that
\begin{equs}
\bigg(\sum_{n_0\geq 0}&\int_{h\in A(0,n_0)} \bigg\| \frac{\big| f_{n_0}(x+h) - \Gamma_{x+h,x}f_{n_0}(x) \big|_\zeta}{\|h\|_\s^{\gamma-\zeta}} \bigg\|^q_{L^p} \frac{dh}{\|h\|_\s^{|\s|} }\bigg)^{\frac1{q}}\\
&\lesssim \sum_{\beta\geq \zeta} \bigg(\sum_{n_0\geq 0} \sum_{h\in \cE_{n_0}^C} \bigg\| \frac{\big| \aver{f}{n_0}(x+h)-\Gamma_{x+h,x} \aver{f}{n_0}(x) \big|_\beta}{2^{-n_0(\gamma-\beta)}} \bigg\|^q_{\ell^p_{n_0}}\bigg)^{\frac1{q}}\;.
\end{equs}
We turn to the third term on the right hand side of (\ref{Eq:DecompTranslBound}). We have
\begin{equs}
\bigg(\sum_{n_0\geq 0}&\int_{h\in A(0,n_0)} \bigg\| \frac{\big| \Gamma_{x+h,x}\big(f_{n_0}(x) - f(x)\big) \big|_\zeta}{\|h\|_\s^{\gamma-\zeta}} \bigg\|^q_{L^p} \frac{dh}{\|h\|_\s^{|\s|} }\bigg)^{\frac1{q}}\\
&\lesssim \sum_{\beta \geq \zeta}\bigg(\sum_{n_0\geq 0}\int_{h\in A(0,n_0)} \bigg\| \frac{\big|f_{n_0}(x) - f(x)\big) \big|_\beta}{2^{-n_0(\gamma-\beta)}} \bigg\|^q_{L^p} \frac{dh}{\|h\|_\s^{|\s|} }\bigg)^{\frac1{q}}\\
&\lesssim \sum_{\beta \geq \zeta}\bigg(\sum_{n_0\geq 0} \bigg\| \sum_{n\geq n_0}\frac{\big|f_{n+1}(x) - f_n(x) \big|_\beta}{2^{-n_0(\gamma-\beta)}} \bigg\|^q_{L^p}\bigg)^{\frac1{q}}\\
&\lesssim \sum_{\beta \geq \zeta}\bigg(\sum_{n_0\geq 0} \bigg(\sum_{n\geq n_0}\bigg\| \frac{\big|f_{n+1}(x) - f_n(x)\big|_\beta}{2^{-n_0(\gamma-\beta)}} \bigg\|_{L^p}\bigg)^q\bigg)^{\frac1{q}}\;.
\end{equs}
At this point, we use (\ref{Eq:BoundIncrLp}) to get the further bound
\begin{equs}
\bigg\| &\frac{\big|f_{n+1}(x) - f_n(x)\big|_\beta}{2^{-n_0(\gamma-\beta)}} \bigg\|_{L^p} \\
&\lesssim \sum_{\delta\geq \beta}2^{-(n-n_0)(\gamma-\beta)} \sum_{h\in\cE_{n+1}^C}\bigg\| \frac{\big|\aver{f}{n}(x) -\Gamma_{x,x+h} \aver{f}{n+1}(x+h)\big|_\delta}{2^{-n(\gamma-\delta)}} \bigg\|_{\ell^p_n}\;.
\end{equs}
Applying Jensen's inequality on the sum over $n\geq n_0$, we deduce that this in turn
is bounded by
\begin{equs}
\sum_{\delta \geq \beta \geq \zeta}&\bigg(\sum_{n_0\geq 0} \sum_{n\geq n_0}2^{-(n-n_0)(\gamma-\beta)} \sum_{h\in\cE_{n+1}^C}\bigg\| \frac{\big|\aver{f}{n}(x) -\Gamma_{x,x+h} \aver{f}{n+1}(x+h)\big|_\delta}{2^{-n(\gamma-\delta)}} \bigg\|_{\ell^p_n}^q\bigg)^{\frac1{q}}\\
&\lesssim \sum_{\delta \geq \zeta}\bigg(\sum_{n\geq 0} \sum_{h\in\cE_{n+1}^C}\bigg\| \frac{\big|\aver{f}{n}(x) -\Gamma_{x,x+h} \aver{f}{n+1}(x+h)\big|_\delta}{2^{-n(\gamma-\delta)}} \bigg\|_{\ell^p_n}^q\bigg)^{\frac1{q}}\;,
\end{equs}
which is of order $\$\bar f\$$ as required. The bound on the first term on the right hand side of (\ref{Eq:DecompTranslBound}) relies on virtually the same argument, so we do not provide the details. This ensures that $f\in\cD^{\gamma}_{p,q}$ and that $\$f\$ \lesssim \$\bar{f}\$$.

Let us finally assume that $\bar{f}$ is obtained from some $f\in\cD^\gamma_{p,q}$ according to the procedure described in the first part of the statement. We aim at showing that the element built with the procedure in the second part of the statement coincides with $f$. To that end, it suffices to show that $f(x)-f_n(x)$ converges to $0$ as $n\rightarrow\infty$. We have
\begin{equs}
|f(x)-f_n(x)|_\zeta &\lesssim\int_{y\in B(x_n,2^{-n})} 2^{n|\s|} |f(x) - \Gamma_{x,y} f(y)|_\zeta dy\\
&\lesssim \int_{y\in B(x,2^{-n+1})} 2^{n|\s|} |f(x) - \Gamma_{x,y} f(y)|_\zeta dy\\
&\lesssim \sum_{\beta\geq\zeta}\int_{h\in B(0,2^{-n+1})} 2^{n|\s|} \|h\|_\s^{\beta-\zeta} |f(x+h) - \Gamma_{x+h,x} f(x)|_\beta dh\;,
\end{equs}
uniformly over all $x\in\R^d$. Let $\bar{q}$ be the conjugate of $q$. We get
\begin{equs}
\big\| |&f(x)-f_n(x)|_\zeta \big\|_{L^p}\\
&\lesssim \sum_{\beta\geq\zeta}\Big\|\int_{h\in B(0,2^{-n+1})} 2^{n|\s|} \|h\|_\s^{\gamma-\zeta} \frac{|f(x+h) - \Gamma_{x+h,x} f(x)|_\beta}{\|h\|_\s^{\gamma-\beta}}dh\Big\|_{L^p}\\
&\lesssim \sum_{\beta\geq\zeta}\int_{h\in B(0,2^{-n+1})} \Big\| \|h\|_\s^{\gamma-\zeta} \frac{|f(x+h) - \Gamma_{x+h,x} f(x)|_\beta}{\|h\|_\s^{\gamma-\beta}}\Big\|_{L^p}\frac{dh}{\|h\|_\s^{|\s|}}\\
&\lesssim \sum_{\beta\geq\zeta}\bigg(\int_{h\in B(0,2^{-n+1})} \Big\| \frac{|f(x+h) - \Gamma_{x+h,x} f(x)|_\beta}{\|h\|_\s^{\gamma-\beta}}\Big\|_{L^p}^q\frac{dh}{\|h\|_\s^{|\s|}}\bigg)^{\frac1{q}}\\
&\qquad\times\bigg(\int_{h\in B(0,2^{-n+1})} \|h\|_\s^{\bar{q}(\gamma-\zeta)}\frac{dh}{\|h\|_\s^{|\s|}}\bigg)^{\frac1{\bar{q}}}
\lesssim 2^{-n(\gamma-\zeta)} \$f\$\;,
\end{equs}
which vanishes as $n\rightarrow\infty$, thus concluding the proof.
\end{proof}

Let us point out again that, as already observed in the proof of Lemma~\ref{Lemma:LocalBoundn}, 
these spaces are \textit{essentially} nested in the sense that, for $f \in \cD^{\gamma}_{p,q}$ (resp.\  
$\aver{f}{} \in \bar{\cD}^{\gamma}_{p,q}$), their projection to $\cT_{<\gamma'}$ lies in $\cD^{\gamma'}_{p,q}$ 
(resp.\ in $\bar{\cD}^{\gamma'}_{p,q}$) whenever $\gamma' < \gamma$. In the case where $q <\infty$
however, this only holds if $\gamma' \notin\cA_\gamma$. This further restriction is a consequence of 
our model being bounded in a H\"older-type norm. In order to avoid this problem, we therefore 
make the following assumption for the remainder of this article.

\begin{assumption}\label{Assumption:Gamma}
The parameter $\gamma$ does not coincide with an element in $\cA$.
\end{assumption}

\section{The reconstruction theorem}\label{Sec:Reconstruction}

Before we turn to the statement of the reconstruction theorem in this context, we
introduce a distance between two modelled distributions $f$ and $\bar{f}$ built from two 
possibly different models $(\Pi,\Gamma)$ and $(\bar{\Pi},\bar{\Gamma})$. Following
\cite[Rem.~3.6]{Hairer2014}, we set
\begin{equs}
{}&\$f,\bar{f}\$ = \Big\| \big| f(x)-\bar{f}(x) \big|_\zeta \Big\|_{L^p} \\
&+ \bigg(\int_{B(0,1)}  \bigg\| \frac{\big| f(x+h) - \bar{f}(x+h)-\Gamma_{x+h,x} f(x) + \bar{\Gamma}_{x+h,x} \bar{f}(x) \big|_\zeta}{\|h\|_\s^{\gamma-\zeta}} \bigg\|^q_{L^p} \frac{dh}{\|h\|_\s^{|\s|}}\bigg)^{\frac1{q}}\;.
\end{equs}
From now on, we also assume that the polynomial regularity structure $(\bar\cA,\bar{\cT},\bar\cG)$ is 
included in the regularity structure under consideration, and that it provides the only elements with 
integer homogeneity. This is not an essential assumption for Theorem~\ref{Th:Reconstruction}, but 
it simplifies its statement.

\begin{theorem}\label{Th:Reconstruction}
Let $(\cA,\cT,\cG)$ be a regularity structure and $(\Pi,\Gamma)$ be a model. Let $\gamma \in \R_+\backslash\N$, and set $\alpha=\min(\cA\backslash\N) \wedge \gamma$. If $q=\infty$, let $\bar\alpha = \alpha$, otherwise take $\bar\alpha < \alpha$. Then, for $\gamma > 0$, there exists a unique continuous linear map $\cR:\cD^\gamma_{p,q}\rightarrow\cB^{\bar\alpha}_{p,q}$ such that
\begin{equ}\label{Eq:BoundRecons}
\bigg\| \Big\| \sup_{\eta \in \BB^r} \frac{\big|\langle \cR f - \Pi_x f(x), \eta_x^\lambda\rangle \big|}{\lambda^{\gamma}} \Big\|_{L^p} \bigg\|_{L^q_\lambda} \lesssim \$ f\$ \|\Pi\| (1+\|\Gamma\|)\;,
\end{equ}
uniformly over all $f\in\cD^\gamma_{p,q}$ and all models $(\Pi,\Gamma)$.

Furthermore, given a second model $(\bar{\Pi},\bar\Gamma)$ and denoting by $\bar\cR$ the associated reconstruction operator, we have
\begin{equs}\label{Eq:BoundRecons2}
{}&\bigg\| \Big\| \sup_{\eta \in \BB^r} \frac{\big|\langle \cR f - \bar\cR \bar{f} - \Pi_x f(x) + \bar{\Pi}_x \bar{f}(x), \eta_x^\lambda\rangle \big|}{\lambda^{\gamma}} \Big\|_{L^p} \bigg\|_{L^q_\lambda}\\
&\lesssim \$ f;\bar{f}\$ \|\Pi\| (1+\|\Gamma\|) + \$ \bar{f}\$ \Big( \|\Pi -\bar{\Pi}\|(1+\|\Gamma\|) + \|\bar{\Pi}\| \|\Gamma-\bar{\Gamma}\| \Big)\;.
\end{equs}
\end{theorem}

Let us comment on the definition of $\alpha$. If the regularity structure has some level of negative homogeneity, then $\alpha$ is taken to be the lowest homogeneity. On the other hand, if the regularity structure consists of usual monomials and of levels with positive, non-integer homogeneity, then $\alpha$ is the lowest non-integer homogeneity. Finally, if the regularity structure consists only of usual monomials, then $\alpha$ is equal to $\gamma$: however, the case where $\gamma\in\N$ has to be treated differently and is not covered by this result (except in a trivial way by using the aforementioned
fact that elements of $\cD^\gamma_{p,q}$ also belong to $\cD^{\gamma'}_{p,q}$ for
$\gamma' < \gamma$).

\begin{remark}
Recall that the definition of a model $(\Pi,\Gamma)$ used in this article is (a global version of) the one introduced in \cite{Hairer2014} which is  modelled on the usual H\"older norms. This is the main reason for
the fact that we need to take $\bar\alpha < \alpha$ in Theorem~\ref{Th:Reconstruction} when $q < \infty$. 
Indeed, assuming that $\alpha = \min \cA < 0$ and writing $\Xi$ for an element in $\CT_\alpha$, 
one easily sees that the modelled distribution $f = \Xi$ does belong to $\cD^\gamma_{p,q}$ (at least locally in space) for all $q \in [1,\infty]$. 
However, one has $\cR f = \Pi_x \Xi \in \cC^\alpha$,
which does not necessarily belong to $\cB^\alpha_{p,q}$ whenever $q < \infty$. 
\end{remark}
\begin{remark}
When the regularity structure consists only of usual monomials, our reconstruction theorem as stated 
asserts that the distribution has regularity $\gamma^-$ if $q<\infty$. The results in the 
following subsection ensure that the regularity is actually $\gamma$ also in this case.
\end{remark}

\subsection{A consequence of the reconstruction theorem}

Recall the notation $\bar\cT$ for the polynomial regularity structure. The following result states that the two spaces $\cB^\gamma_{p,q}$ and $\cD^\gamma_{p,q}(\bar \cT)$ are the same, which justifies our
definitions.
\begin{proposition}\label{Prop:OnetoOne}
Take $\gamma \in \R_+ \backslash\N$ and consider the polynomial regularity structure $\bar\cT$. Then, the reconstruction operator $\cR$ is a isomorphism between $\cD^\gamma_{p,q}(\bar \cT)$ and $\cB^\gamma_{p,q}$.
\end{proposition}

This result is the consequence of the following two lemmas. The first one shows that $\cR$ is injective.

\begin{lemma}\label{Lemma:ReconsPoly}
Let $f\in\cD^\gamma_{p,q}(\bar \cT)$. Then, for every $k\in\N^d$ such that $|k|_\s < \gamma$, the map $x\mapsto k! f_k(x)$ coincides (as a distribution) with the $k$-th derivative of $\cR f$, and one
has $\cR f = f_0 \in \cB^\gamma_{p,q}$.
\end{lemma}
\begin{proof}
Let $k\in\N^d$ be such that $|k| < \gamma$. A careful inspection of the proof of the uniqueness part of the reconstruction theorem yields that there is at most one distribution $\xi^{(k)}$ on $\R^d$ such that
\begin{equ}\label{Eq:PartialkR}
\bigg\| \Big\| \sup_{\eta \in \BB^{r+|k|}} \frac{\big|\langle \xi^{(k)} - \partial^k\Pi_x f(x), \eta_x^\lambda\rangle \big|}{\lambda^{\gamma-|k|}} \Big\|_{L^p} \bigg\|_{L^q_\lambda} <\infty\;.
\end{equ}
Since $\partial^k (\eta^\lambda)=\lambda^{-|k|} (\partial^k \eta)^\lambda$ and since $\partial^k \eta \in \BB^r$ whenever $\eta \in \BB^{r+|k|}$, the reconstruction theorem ensures that $\partial^k \cR f$ satisfies such a bound.

Let us now set $\xi^{(k)}(x) = k! f_k(x)$. Since $x\mapsto \xi^{(k)}(x)$ belongs to $L^p(\R^d)$, it defines a distribution on $\R^d$. Furthermore, we have the identity
\begin{equ}
\big(\xi^{(k)} -\partial^k\Pi_x f(x)\big)(y) = k! \cQ_k\big(f(y)-\Gamma_{y,x} f(x)\big)\;.
\end{equ}
Since $f\in\cD^\gamma_{p,q}(\bar T)$, we deduce that (\ref{Eq:PartialkR}) holds for our choice of $\xi^{(k)}$ and consequently, $\xi^{(k)}$ coincides with $\partial^k \cR f$.

To show that $\cR f \in \cB^\gamma_{p,q}$, we first note that
the first bound of \eqref{Eq:BoundDefBesovPos} with $\xi = f_0$ follows immediately from the fact that 
$f_0 \in L^p$ by the definition of $\cB^\gamma_{p,q}$.
The second bound with $\xi = \cR f$ on the other hand follows immediately from
\eqref{Eq:BoundRecons} since $\scal{\Pi_x f(x),\eta_x^\lambda} = 0$ for $\eta \in \BB^r_{\lfloor \gamma\rfloor}(\R^d)$.
\end{proof}

The second lemma constructs the continuous inverse of $\cR$.
\begin{lemma}\label{Prop:Inverse}
There exists a continuous injection $\iota$ from $\cB^\gamma_{p,q}$ into $\cD^\gamma_{p,q}(\bar{\cT})$ such that $\cR\iota \xi = \xi$ for all $\xi\in\cB^\gamma_{p,q}$.
\end{lemma}
\begin{proof}
Let $\rho:\R^d\rightarrow\R_+$ be a smooth, even function, supported in the unit ball of $\R^d$, that integrates to $1$. For simplicity, we write $\rho^n_x(y)$ instead of $\rho^{2^{-n}}_x(y)$. Let $\xi\in\cB^\gamma_{p,q}$. For every $n\geq 0$, every $x\in\Lambda_n$ and every $k\in\N^d$ such that $|k|<\gamma$, we set
\begin{equ}
\aver{f}{n}_k(x) = \langle \partial^k \xi , P^{\lfloor \gamma \rfloor}_{k,x}(\rho^{n},\cdot)\rangle\;,
\end{equ}
where
\begin{equ}
P^q_{k,x}(\eta,y) = \sum_{\ell\in\N^d:|k+\ell| \leq q} (-1)^\ell \partial^\ell_y\Big(\eta(y-x) \frac{(x-y)^\ell}{k!\ell !}\Big)\;.
\end{equ}
for any $q\in\N$, any $k\in\N^d$ such that $|k|\leq q$ and any smooth function $\eta$. This definition of $\bar{f}$ may not seem obvious at first sight, but it can actually be guessed easily from (\ref{Eq:ffbar}) upon replacing $2^{n|\s|} \tun_{B(x,2^{-n})}$ by $\rho^n_x$, combined with the action of $\Gamma_{x,y}$ on the polynomial regularity structure.

We aim at showing that $\bar{f} \in \bar{\cD}^\gamma_{p,q}$. The local bound is easy to check since
\begin{equ}
\Big\| \big| \aver{f}{0} (x) \big|_\zeta \Big\|_{\ell^p_0} \lesssim \|\xi\|_{B^\gamma_{p,q}}\;.
\end{equ}
Regarding the translation and consistency bounds, we introduce for all $h\in\cE_n$, all $x\in\Lambda_n$ and all $n\geq 0$ the functions
\begin{equ}
\Psi^q_k:y\mapsto P^{q}_{k,x+h}(\rho^{n},y) - \sum_{\ell\in\N^d:|k+\ell| \leq q} (-h)^\ell\frac{(k+\ell)!}{k! \ell!}\partial^\ell_y P^{q}_{k+\ell,x}(\rho^{n},y)\;,
\end{equ}
and
\begin{equ}
\Phi^q_k:y\mapsto P^{q}_{k,x}(\rho^{n},y)-P^{q}_{k,x}(\rho^{n+1},y)\;.
\end{equ}
These functions have been defined so that the following two identities hold
\begin{equs}
\cQ_k\big(\aver{f}{n}(x) - \aver{f}{n+1}(x)\big) &= \langle \partial^k \xi , \Phi^{\lfloor \gamma\rfloor}_k \rangle\;,\\
\cQ_k\Big(\aver{f}{n}(x+h) - \Gamma_{x+h,x}\aver{f}{n}(x)\Big) &= \langle \partial^k \xi , \Psi^{\lfloor \gamma\rfloor}_k \rangle\;.
\end{equs}
Both $\Phi^q_k$ and $\Psi^q_k$ are smooth functions, compactly supported in a ball centred at $x$ and of radius of order $2^{-n}$. Assume that they both annihilate all polynomials of scaled degree lower than $q-|k|_\s$ and recall that $\partial^k \xi$ belongs to $\cB^{\gamma-|k|}_{p,q}$. We then easily obtain the translation and consistency bounds by applying Definition \ref{Def:Besov}.

It therefore remains to prove that $\Phi^q_k$ and $\Psi^q_k$ do indeed annihilate polynomials 
of degree $q-|k|_\s$. First of all, a simple integration by parts ensures that
\begin{equ}
\int_y P^{\lfloor \gamma\rfloor}_{k,x}(\rho^{n},y) dy = \frac1{k!} \int_y \rho^{n}_{x}(y) dy = \frac1{k!} \int_y \rho(y) dy\;,
\end{equ}
so that $\Phi^q_k$ and $\Psi^q_k$ annihilate constants.
Second, we prove by recursion on $q$ that the following property holds true. For every $k\in\N^d$, the function $y\mapsto P^{q}_{k,x}(\rho^{n},y)$ kills all monomials $(y-x)^m$ with $m\in\N^d$, $m\ne 0$ and $|m+k|_\s \leq q$. Once this property is established, one easily deduces that $\Phi^q_k$ annihilates all polynomials with a scaled degree which is non-zero and lower than $q-|k|_\s$. A similar recursion yields the desired property for $\Psi^q_k$, which is left to the reader. 

First, we check that the property is true at rank $q=|k|_\s+1$. Take $m\in \N^d$ such that $|m|_\s=1$. We have
\begin{equs}
{}&k!\int_y (y-x)^m P^{q}_{k,x}(\rho^{n},y) dy\\
&= \int_y (y-x)^m \rho^n_x(y) dy + \sum_{\ell\in\N^d: |\ell|=1} \int_y (y-x)^m (-1)^\ell \partial_y^\ell\Big((x-y)^\ell \rho^n_x(y) \Big) dy\\
&= \int_y (y-x)^m \rho^n_x(y) dy + \sum_{\ell\in\N^d: |\ell|=1} \int_y \partial_y^\ell((y-x)^m) (x-y)^\ell \rho^n_x(y) dy\;.
\end{equs}
Since $|\ell|_\s = |m|_\s = 1$, the only non-zero contribution in the second term on the right hand side  comes from $\ell=m$. Hence, the sum of the two terms vanishes and the property is true at rank $q=|k|_\s+1$. Assume now that it holds at rank $q-1$, for some $q\geq |k|_\s+2$. Observe that
\begin{equ}
P^{q}_{k,x}(\rho^{n},y) = P^{q-1}_{k,x}(\rho^{n},y) + \sum_{\ell\in\N^d:|k+\ell|=q} (-1)^\ell \partial^\ell_y\Big(\rho^n_x(y) \frac{(x-y)^\ell}{k!\ell !}\Big)\;.
\end{equ}
By the recursion hypothesis, we know that the first term on the right kills $(y-x)^m$ for all $m\in\N^d$ such that $|k|_\s<|m+k|_\s<q$. A simple integration by parts then shows that the second term satisfies the same property. Furthermore, for all $m\in\N^d$ such that $|m+k|_\s=q$, an integration by parts yields (notice that we indeed impose $\ell \le m$ in the sum at the second line, and not $|\ell| \le |m|$)
\begin{equs}
\int_y (y-x)^m P^{q}_{k,x}(\rho^{n},y) dy &= \sum_{\ell\in\N^d:|k+\ell| \leq q} \int_y (y-x)^m (-1)^\ell \partial^\ell_y\Big(\rho^n_x(y) \frac{(x-y)^\ell}{\ell ! k!}\Big) dy\\
&= \sum_{\ell\in\N^d:\ell \leq m} \int_y (y-x)^m \rho^n_x(y)  (-1)^\ell \frac{m!}{\ell ! (m-\ell)! k!} dy\\
&=\frac1{k!} \int_y (y-x)^m \rho^n_x(y) dy \prod_{i=1}^d \sum_{\ell_i\in\N:\ell_i \leq m_i} (-1)^{\ell_i} {{m_i}\choose{\ell_i}}\;,
\end{equs}
which vanishes by the binomial formula, thus completing the proof of the recursion.

We have shown that $\bar{f}\in\bar{\cD}^\gamma_{p,q}(\bar\cT)$. Applying the second part of Theorem \ref{Th:EqDgamma}, we obtain an element $f\in\cD^\gamma_{p,q}(\bar{\cT})$ and we naturally set $\iota\xi :=f$. A careful look at the proof of the theorem yields that $f_0(x)$ is the limit in $L^p(dx)$ of the sequence $\cQ_0\Gamma_{x,x_n} \aver{f}{n}(x_n)$, where $x_n$ is the nearest point of $x$ on the grid $\Lambda_n$. Lemma \ref{Lemma:ReconsPoly} ensures that $\cR f(x) = f_0(x)$ in $L^p(dx)$. Furthermore, by Lemma \ref{Lemma:BesovFunction} we know that $\xi(x)$ is the limit in $L^p(dx)$ of the sequence $\langle \xi,\rho^n_x\rangle$. Consequently, to prove the identity $\cR\iota\xi = \xi$, it suffices to show that $\langle \xi,\rho^n_x\rangle-\cQ_0\Gamma_{x,x_n} \aver{f}{n}(x_n)$ converges to $0$ in $L^p$ as $n\rightarrow\infty$. We observe that
\begin{equ}
\langle \xi,\rho^n_x\rangle-\cQ_0\Gamma_{x,x_n} \aver{f}{n}(x_n) = \langle \xi,\phi^n_x\rangle\;,
\end{equ}
where
\begin{equ}
\phi^n_x(y) = \rho^n_x(y) - \sum_{\ell\in\N^d:|\ell|<\gamma}(x_n-x)^\ell \partial^\ell_y P^{\lfloor\gamma\rfloor}_{\ell,x_n}(\rho^n,y)\;,
\end{equ}
which is a smooth function, supported in a ball of order $2^{-n}$ around $x$, with a scaling behaviour of order $2^{-n}$ and that kills the constants. Since $\xi \in \cB^\epsilon_{p,\infty}$ for some $\epsilon\in (0,1)$, the definition of that space ensures that the $L^p$ norm of $\langle \xi,\phi^n_x\rangle$ vanishes as $n\rightarrow\infty$, thus concluding the proof.
\end{proof}

\subsection{Proof of the reconstruction theorem}

We start with a convergence criterion in $\cB^\alpha_{p,q}$ with $\alpha < 0$, which is an adaptation of \cite[Thm~3.23]{Hairer2014}. Recall the wavelet analysis introduced in Section \ref{Section:Wavelet}. For a sequence of values $A^n_x, x\in \Lambda_n$, consider the distribution
\begin{equ}
\xi_n = \sum_{x\in\Lambda_n} A^n_x \phi^n_x \in V_n\;,
\end{equ}
and set $\delta A^n_x = \langle \xi_{n+1}-\xi_n , \phi^n_x\rangle$.
\begin{proposition}\label{Prop:SewingNeg}
Let $\alpha < 0$ and $\gamma>0$. Assume that
\begin{equs}\label{Eq:CondSewing}
\sup_{n\geq 0} \Big\| \frac{A^n_x}{2^{-n\alpha-n\frac{|\s|}{2}}} \Big\|_{\ell^p_n} \lesssim 1 \;,\quad
\bigg\| \Big\| \frac{\delta A^n_x}{2^{-n\gamma-n\frac{|\s|}{2}}} \Big\|_{\ell^p_n} \bigg\|_{\ell^q} \lesssim 1 \;.\\
\end{equs}
Then, as $n\rightarrow\infty$, $\xi_n\rightarrow \xi$ in $\cB^{\bar \alpha}_{p,q}$ for all $\bar \alpha < \alpha$. Furthermore, when $q=\infty$ the limit $\xi$ belongs to $\cB^{\alpha}_{p,q}$.
\end{proposition}
\begin{proof}
For every $n\geq 0$, we write $\xi_{n+1}-\xi_n = g_n + \delta \xi_n$ where $g_n\in V_n$ and $\delta \xi_n \in V_n^\perp$, where $V_n^\perp$ is defined as the orthogonal complement of $V_n$ in $V_{n+1}$. We treat separately the contributions of these two terms. We start with $g_n$. For all $n\geq 0$, we have
\begin{equs}\label{Eq:Projgn}
\| g_n \|_{\cB^{\bar\alpha}_{p,q}} = \Big\| \langle g_n,\phi^0_x\rangle \Big\|_{\ell^p_0} + \sum_{\psi\in\Psi}\bigg\|\Big\|\frac{\langle g_n , \psi^m_x\rangle}{2^{-m(\bar\alpha+\frac{|\s|}{2})}} \Big\|_{\ell^p_m} \bigg\|_{\ell^q(m\geq 0)}\;.
\end{equs}
Notice that, whenever $m\geq n$, the corresponding terms in the right hand side vanish. On the other hand, for every $m<n$, we observe that $\big| \langle \phi^n_y , \psi^m_x\rangle \big| \lesssim 2^{-(n-m){|\s| / 2}}$ uniformly over all $x,y$. Actually this inner product vanishes as soon as $\|y-x\|_{\s} > C 2^{-m}$ for some constant $C>0$ that depends on the size of the supports of $\phi,\psi$. Using Jensen's inequality at the second line, we thus get
\begin{equs}
\Big\|\frac{\langle g_n , \psi^m_x\rangle}{2^{-m(\bar\alpha+\frac{|\s|}{2})}} \Big\|_{\ell^p_m} &\lesssim \Big\| \sum_{\substack{y\in\Lambda_n\\\|y-x\|_\s\leq C 2^{-m}}}\frac{|\delta A^n_y|}{2^{-m(\bar\alpha+|\s|)}} 2^{-n\frac{|\s|}{2}} \Big\|_{\ell^p_m}\\
&\lesssim 2^{-n\gamma+m\bar\alpha} \Big(\sum_{x\in\Lambda_m}\sum_{\substack{y\in\Lambda_n\\\|y-x\|_\s \leq C 2^{-m}}} 2^{-n|\s|} \Big| \frac{\delta A^n_y}{2^{-n(\gamma+\frac{|\s|}{2})}}  \Big|^p\Big)^{\frac{1}{p}}\\
&\lesssim 2^{-n\gamma+m\bar\alpha} \Big(\sum_{y\in\Lambda_n} 2^{-n|\s|} \Big| \frac{\delta A^n_y}{2^{-n(\gamma+\frac{|\s|}{2})}}  \Big|^p\Big)^{\frac{1}{p}}\;,
\end{equs}
uniformly over all $m<n$. Recall that $\bar\alpha < 0$. We obtain
\begin{equ}
\bigg\|\Big\|\frac{\langle g_n , \psi^m_x\rangle}{2^{-m(\bar\alpha+\frac{|\s|}{2})}} \Big\|_{\ell^p_m} \bigg\|_{\ell^q(m\geq 0)} \lesssim 2^{-n\gamma} \Big\| \frac{\delta A^n_x}{2^{-n\gamma-n\frac{|\s|}{2}}} \Big\|_{\ell^p_n}\;,
\end{equ}
uniformly over all $n\geq 0$. Similar calculations yield the same bound for the first term on the right hand side of (\ref{Eq:Projgn}). Consequently, using H\"older's inequality with $q$ and its conjugate exponent $\bar q$, we have
\begin{equs}
\Big\| \sum_{n_0 \leq n \leq n_1} g_n \Big\|_{\cB^{\bar\alpha}_{p,q}} &\leq \sum_{n_0 \leq n \leq n_1}\| g_n \|_{\cB^{\bar\alpha}_{p,q}}\\
&\lesssim \bigg\|\Big\| \frac{\delta A^n_x}{2^{-n\gamma-n\frac{|\s|}{2}}} \Big\|_{\ell^p_n}\bigg\|_{\ell^q(n_0 \leq n \leq n_1)} \Big\| 2^{-n\gamma} \Big\|_{\ell^{\bar q}(n_0 \leq n \leq n_1)}\\
&\lesssim 2^{-n_0\gamma}\;,
\end{equs}
so that $\sum_{n\geq 0} g_n$ converges in $\cB^{\bar\alpha}_{p,q}$. Notice that one only needs $\bar\alpha < 0$ for the arguments to apply.

We turn to $\delta \xi_n$. By Proposition \ref{PropCharact} and since $\delta \xi_n \in V_n^\perp$, we have
\begin{equ}
\Big\| \sum_{n_0 \leq n \leq n_1} \delta \xi_n \Big\|_{\cB^{\bar\alpha}_{p,q}} =  \Big(\sum_{n_0 \leq n \leq n_1} \| \delta \xi_n \|_{\cB^{\bar\alpha}_{p,q}}^q \Big)^{\frac{1}{q}}\;.
\end{equ}
Using Jensen's inequality at the second line, we get
\begin{equs}
\| \delta \xi_n \|_{\cB^{\bar\alpha}_{p,q}} &\lesssim \sum_{\psi\in\Psi}\Big\| \frac{\langle \delta \xi_n , \psi^n_x \rangle}{2^{-n(\bar\alpha+\frac{|\s|}{2})}}\Big\|_{\ell^p_n}\\
&\lesssim \Big\| \sum_{\substack{y\in\Lambda_{n+1}\\\|y-x\|_\s \leq C 2^{-n}}}\frac{|A^{n+1}_y|}{2^{-n(\bar\alpha+\frac{|\s|}{2})}}\Big\|_{\ell^p_n}\\
&\lesssim \Big( \sum_{x\in\Lambda_n}\sum_{\substack{y\in\Lambda_{n+1}\\\|y-x\|_\s \leq C 2^{-n}}} 2^{-n|\s|} \Big(\frac{|A^{n+1}_y|}{2^{-n(\bar\alpha+\frac{|\s|}{2})}}\Big)^p \Big)^{\frac{1}{p}}\\
&\lesssim 2^{-n(\alpha - \bar\alpha)}\Big(\sum_{y\in\Lambda_{n+1}} 2^{-(n+1)|\s|} \Big(\frac{|A^{n+1}_y|}{2^{-(n+1)(\alpha+\frac{|\s|}{2})}}\Big)^p \Big)^{\frac{1}{p}} \;,
\end{equs}
uniformly over all $n\geq 0$. Therefore, as soon as $\bar\alpha < \alpha$ we get
\begin{equs}
\Big(\sum_{n_0 \leq n \leq n_1} \| \delta \xi_n \|_{\cB^{\bar\alpha}_{p,q}}^q \Big)^{\frac{1}{q}} &\lesssim \bigg( \sum_{n_0 \leq n \leq n_1+1} 2^{-n(\alpha-\bar\alpha)q} \Big\| \frac{A^n_x}{2^{-n(\alpha+\frac{|\s|}{2})}} \Big\|_{\ell^p_n}^q \bigg)^{\frac{1}{q}}\\
&\lesssim 2^{-n_0(\alpha-\bar\alpha)} \sup_{n_0 \leq n \leq n_1+1} \Big\| \frac{A^n_x}{2^{-n(\alpha+\frac{|\s|}{2})}} \Big\|_{\ell^p_n}\;,
\end{equs}
so that $\sum_{n\geq 0} \delta \xi_n$ converges in $\cB^{\bar\alpha}_{p,q}$. In the particular case $q=\infty$, we have
\begin{equ}
\Big\| \sum_{n \leq n_1} \delta \xi_n \Big\|_{\cB^{\alpha}_{p,q}} \lesssim \sup_{n\leq n_1}\Big\|  \delta \xi_n \Big\|_{\cB^{\alpha}_{p,q}}
\lesssim \sup_{n\geq 0} \Big\| \frac{ A^n_x}{2^{-n(\alpha+\frac{|\s|}{2})}} \Big\|_{\ell^p_n}\;,
\end{equ}
so that the limit belongs to $\cB^{\alpha}_{p,\infty}$.
\end{proof}
Before we proceed to the proof of the reconstruction theorem, let us introduce the following notation. Let $L^q_{n_0}$ be the space of all measurable functions $g:(2^{-n_0-1}, 2^{-n_0}]\rightarrow \R$ such that
$$ \Big(\int_{(2^{-n_0-1}, 2^{-n_0}]} |g(\lambda)|^q \,\frac{d\lambda}{\lambda}\Big)^{\frac1{q}} < \infty\;.$$
Observe that for any function $g:(0,1]\rightarrow\R$ we have the identity $\big\| g \big\|_{L^q_\lambda} = \big\| \big\| g \big\|_{L^q_{n_0}} \big\|_{\ell^q(n_0\geq 0)}$. 
\begin{proof}[of Theorem \ref{Th:Reconstruction}]
From now on, the symbol $\zeta$ is implicitly taken in the set of homogeneities $\cA$, and we omit to write the corresponding sum over all $\zeta\in \cA$ in order to alleviate the notations. Let $f\in \cD^\gamma_{p,q}$ and take $\aver{f}{}$ as defined in (\ref{Eq:ffbar}). For all $n\geq 0$ and $x\in\Lambda_n$, we set
\begin{equ}
A^n_x := \langle \Pi_x \aver{f}{n}(x) , \phi^n_x \rangle\;\;,
\end{equ}
and we define $\cR_n f = \sum_{x\in\Lambda_n} A^n_x \phi^n_x$.\\

\noindent\textit{Step 1: convergence.} We show that $\cR_n f$ converges to an element $\cR f$ in $\cB^{\bar \alpha}_{p,q}$ for all $\bar\alpha < \alpha\wedge 0$. To that end, let us check that the conditions of Proposition~\ref{Prop:SewingNeg} 
are satisfied with our present choice of $A^n_x$'s and with $\alpha$ replaced by $\alpha':=\alpha\wedge(-\epsilon)$ for some arbitrary $\epsilon>0$. The first condition of (\ref{Eq:CondSewing}) is a direct consequence of the local bound on $\bar f$. To obtain the second bound, we write
\begin{equs}\label{Eq:deltaAn}
\delta A^n_x = \sum_{y\in\Lambda_{n+1}} \big\langle \Pi_y\big(\aver{f}{n+1}(y) - \Gamma_{y,x} \aver{f}{n}(x)\big) , \varphi^{n+1}_y\big\rangle \langle \varphi^{n+1}_y,\varphi^n_x\rangle\;,
\end{equs}
so that
\begin{equs}
\bigg\|\Big\| \frac{\delta A^n_x}{2^{-n\gamma-n\frac{|\s|}{2}}} \Big\|_{\ell^p_n}\bigg\|_{\ell^q} \lesssim \bigg\|\Big\| \sum_{\substack{y\in\Lambda_{n+1}\\\|y-x\|_\s \leq C 2^{-n}}} \frac{\big|\aver{f}{n+1}(y) - \Gamma_{y,x} \aver{f}{n}(x)\big|_\zeta}{2^{-n(\gamma-\zeta)}} \Big\|_{\ell^p_n}\bigg\|_{\ell^q}\;,
\end{equs}
which is bounded by a term of order $\$ \bar f\$$ thanks to Remark~\ref{Rmk:Consistency}. Therefore, as claimed, $\cR_n f$ converges to some element $\cR f$ which belongs to $\cB^{\bar\alpha}_{p,q}$ for any $\bar\alpha < \alpha'=\alpha\wedge(-\epsilon)$, and therefore, for any $\bar\alpha < \alpha \wedge 0$.\\

\noindent\textit{Step 2: reconstruction bound.} Let us now show the bound (\ref{Eq:BoundRecons}). Given $\lambda\in (0,1]$, let $n_0$ be the largest integer such that $2^{-n_0} \geq \lambda$. Recall that $\cP_{n}$ is the projection onto $V_n$ and $\cP^\perp_n$ the projection onto $V_n^\perp$ (the orthogonal complement of $V_n$ in $V_{n+1}$). For every $n_0\geq 0$, we will use the following decomposition
\begin{equ}\label{Eq:DecompoBoundRecons}
\cR f - \Pi_x f(x)= \big(\cR_{n_0} f - \cP_{n_0}\Pi_x f(x)\big) + \sum_{n\geq n_0} \big(\cR_{n+1} f - \cR_n f -\cP^\perp_n \Pi_x f(x)\big) \;.
\end{equ}
We treat separately the contributions of the two terms on the right hand side. We start with the contribution on $V_{n_0}$:
\begin{equs}
{}&\cR_{n_0} f - \cP_{n_0}\Pi_x f(x)\\
&= \sum_{y\in \Lambda_{n_0}} \big( A^{n_0}_y - \langle \Pi_x f(x) , \phi^{n_0}_y \rangle \big) \phi^{n_0}_y\\
&= \sum_{y\in \Lambda_{n_0}} \int_{z\in B(y,2^{-n_0})} 2^{n_0|\s|} \big\langle \Pi_z(f(z)-\Gamma_{z,x} f(x) , \phi^{n_0}_y \big\rangle dz\, \phi^{n_0}_y\;.
\end{equs}
There exists $C>0$ such that, uniformly over all $n_0\geq 0$, we have
\begin{equs}
{}& \bigg\|\Big\|\sup_{\eta\in\BB^r}\frac{\big|\langle \cR_{n_0} f - \cP_{n_0}\Pi_x f(x) , \eta^\lambda_x \rangle\big|}{\lambda^\gamma} \Big\|_{L^p(dx)} \bigg\|_{L^q_{n_0}}\\
&\lesssim \bigg\|\Big\| \sum_{\substack{y\in \Lambda_{n_0}\\\|y-x\|_\s \leq C \lambda}}\int_{z\in B(x,C 2^{-n_0})} 2^{n_0|\s|} \frac{|f(z)-\Gamma_{z,x} f(x)|_{\zeta} }{\|z-x\|_\s^{\gamma-\zeta}}dz  \Big\|_{L^p(dx)}\bigg\|_{L^q_{n_0}}\\
&\lesssim \Big\| \int_{h\in B(0,C 2^{-n_0})} 2^{n_0|\s|} \frac{|f(x+h)-\Gamma_{x+h,x} f(x)|_{\zeta} }{\|h\|_\s^{\gamma-\zeta}} dh \Big\|_{L^p(dx)}\\
&\lesssim \int_{h\in B(0,C 2^{-n_0})} 2^{n_0|\s|} \Big\| \frac{|f(x+h)-\Gamma_{x+h,x} f(x)|_{\zeta} }{\|h\|_\s^{\gamma-\zeta}} \Big\|_{L^p(dx)} dh\;.
\end{equs}
Using Jensen's inequality, we deduce that the $\ell^q(n_0\geq 0)$-norm of the last expression is bounded by a term of order $\$f\$$ as required.

We now treat the second term of (\ref{Eq:DecompoBoundRecons}). To that end, we write $\cR_{n+1} f - \cR_n f = g_n + \delta f_n$ where $g_n \in V_n$ and $\delta f_n \in V_n^\perp$. Then, we have
\begin{equs}
{}&\langle \delta f_n - \cP_n^\perp \Pi_x f(x) , \eta^\lambda_x \rangle \\
&= \sum_{y\in\Lambda_{n+1}} \sum_{z\in\Lambda_n} \big(A^{n+1}_y - \langle \Pi_x f(x) , \phi^{n+1}_y \rangle \big) \langle \phi^{n+1}_y , \psi^n_z\rangle \langle \psi^n_z , \eta^\lambda_x\rangle\;.
\end{equs}
We have $|\langle \phi^{n+1}_y , \psi^n_z\rangle| \lesssim 1$ and $|\langle \psi^n_z , \eta^\lambda_x\rangle| \lesssim 2^{-n(\frac{|\s|}{2}+r)} \lambda^{-(|\s|+r)}$ uniformly over all the parameters. To get the second bound, we used the fact that $\psi$ annihilates polynomials of any order up to $r$. Actually, the first, resp.~second, inner product vanishes as soon as $\|y-z\|_\s \leq C 2^{-n}$, resp.~$\|z-x\|_\s \leq C \lambda$, for some constant $C>0$ depending on the sizes of the supports of $\phi,\psi$. Given the expression of $A^{n+1}_y$, some simple calculations yield the existence of $C'>0$ such that
\begin{equs}
{}& \bigg\|\Big\|\sup_{\eta\in\BB^r}\frac{\big|\langle \delta f_n - \cP_n^\perp \Pi_x f(x) , \eta^\lambda_x \rangle\big|}{\lambda^\gamma} \Big\|_{L^p(dx)} \bigg\|_{L^q_{n_0}}\\
&\lesssim \bigg\|\Big\| \sum_{\substack{z\in\Lambda_n\\ \|z-x\|_\s \leq C\lambda}} \sum_{\substack{y\in\Lambda_{n+1}\\\|y-z\|_\s\leq C 2^{-n}}} \frac{|A^{n+1}_y - \langle \Pi_x f(x) , \phi^{n+1}_y \rangle|}{\lambda^{\gamma+r+|\s|}} 2^{-n(\frac{|\s|}{2}+r)} \Big\|_{L^p(dx)} \bigg\|_{L^q_{n_0}} \\
&\lesssim \bigg\|\Big\| \!\!\!\sum_{\substack{y\in\Lambda_{n+1}\\\|y-x\|_\s\leq 2C \lambda}}\!\!\! \int_{u\in B(y,2^{-(n+1)})}\!\!\!\! 2^{(n+1)|\s|}|f(u)-\Gamma_{u,x} f(x)|_\zeta\frac{2^{-n(|\s|+\zeta+r)}}{2^{-n_0(\gamma+r+|\s|)}} \Big\|_{L^p(dx)} \bigg\|_{L^q_{n_0}} \\
&\lesssim \Big\| \int_{h\in B(0,C' 2^{-n_0})}2^{n_0|\s|}\frac{|f(x+h)-\Gamma_{x+h,x} f(x)|_\zeta}{\|h\|_\s^{\gamma-\zeta}} dh  \Big\|_{L^p(dx)}2^{-(n-n_0)(\zeta+r)} \;,
\end{equs}
uniformly over all $n_0$. Since $r>|\alpha|$, the sum over all $n\geq n_0$ of the last expression converges. Then, taking the $\ell^q(n_0)$-norm, one gets
\begin{equs}
{}& \bigg\|\Big\| \sum_{n\geq n_0}\frac{\langle \delta f_n - \cP_n^\perp \Pi_x f(x) , \eta^\lambda_x \rangle}{\lambda^\gamma} \Big\|_{L^p(dx)} \bigg\|_{L^q_{\lambda}}\\
&\lesssim \bigg\|\int_{h\in B(0,C' 2^{-n_0})}2^{n_0|\s|}\Big\| \frac{|f(x+h)-\Gamma_{x+h,x} f(x)|_\zeta}{|h|^{\gamma-\zeta}}  \Big\|_{L^p(dx)}dh \bigg\|_{\ell^q(n_0)}\\
&\lesssim \bigg(\int_{h\in B(0,1)}\Big\| \frac{|f(x+h)-\Gamma_{x+h,x} f(x)|_\zeta}{|h|^{\gamma-\zeta}}  \Big\|_{L^p(dx)}^q \frac{dh}{|h|^{|\s|}} \bigg)^{\frac1{q}}\;,
\end{equs}
as required.

Finally we treat the contribution of $g_n$. First, using (\ref{Eq:deltaAn}) we have
\begin{equ}
|\langle g_n , \eta^\lambda_x \rangle| \lesssim \sum_{\substack{y\in\Lambda_n\\\|y-x\|_\s\leq C\lambda}}\sum_{\substack{z\in\Lambda_{n+1}\\\|z-y\|_\s\leq C2^{-n}}} 2^{-(n-n_0)|\s|-n\zeta} |\aver{f}{n+1}(z)-\Gamma_{z,y} \aver{f}{n}(y)|_\zeta\;,
\end{equ}
uniformly over all $n\geq n_0 \geq 0$, and all $x\in\R^d$. Therefore, the triangle inequality at the second line and Jensen's inequality at the third line yield
\begin{equs}
{}&\bigg\|\Big\| \sum_{n\geq n_0}\sup_{\eta\in\BB^r}\frac{\big|\langle g_n , \eta^\lambda_x \rangle\big|}{\lambda^\gamma}\Big\|_{L^p}\bigg\|_{L^q_{n_0}}\\
&\lesssim \sum_{n\geq n_0} 2^{n_0\gamma} \Big\| \sum_{\substack{y\in\Lambda_n\\\|y-x\|_\s\leq C2^{-n_0}}}\sum_{\substack{z\in\Lambda_{n+1}\\\|z-y\|_\s\leq C2^{-n}}}  2^{-(n-n_0)|\s|-n\zeta} |\aver{f}{n+1}(z)-\Gamma_{z,y} \aver{f}{n}(y)|_\zeta \Big\|_{L^p} \\
&\lesssim \sum_{n\geq n_0} 2^{-(n-n_0)\gamma} \sum_{h\in \cE^C_{n+1}} \Big(\sum_{y\in\Lambda_n} 2^{-n|\s|} \Big(\frac{|\aver{f}{n+1}(y+h)-\Gamma_{y+h,y} \aver{f}{n}(y)|_\zeta}{2^{-n(\gamma-\zeta)}}\Big)^p \Big)^{\frac1{p}} \;,
\end{equs}
uniformly over all $n_0$. Therefore, using Jensen's inequality, we get
\begin{equs}
{}&\bigg\|\bigg\|\Big\| \sum_{n\geq n_0}\frac{\langle g_n , \eta^\lambda_x \rangle}{\lambda^\gamma}\Big\|_{L^p}\bigg\|_{L^q_{n_0}}\bigg\|_{\ell^q(n_0\geq 0)}\\
&\lesssim \bigg(\sum_{n_0\geq 0} \sum_{n\geq n_0} 2^{-(n-n_0)\gamma} \sum_{h\in \cE^C_n} \Big\| \frac{|\aver{f}{n+1}(x+h)-\Gamma_{x+h,x} \aver{f}{n}(x)|_\zeta}{2^{-n(\gamma-\zeta)}}\Big\|_{\ell^p_{n}}^q\bigg)^{\frac1{q}}\\
&\lesssim \bigg(\sum_{n\geq 0} \sum_{h\in \cE^C_n}\Big\| \frac{|\aver{f}{n+1}(x+h)-\Gamma_{x+h,x} \aver{f}{n}(x)|_\zeta}{2^{-n(\gamma-\zeta)}}\Big\|_{\ell^p_{n}}^q\bigg)^{\frac1{q}}\;,
\end{equs}
which is bounded by the norm of $f$. This concludes the proof of the reconstruction bound.\\

\noindent\textit{Step 3: improved regularity.} So far, we have showed that $\cR f$ belongs to $\cB^{\bar{\alpha}}_{p,q}$ with $\bar{\alpha} < \alpha \wedge 0$. When $\alpha$ is negative, this is the regularity that we are aiming for while in the other case this is worse than what the statement of the theorem asserts. 
However, the reconstruction bound (\ref{Eq:BoundRecons}), that we established at the second step, allows one to recover the asserted regularity. Indeed, recall that when $\alpha > 0$ the regularity structure only contains non-negative homogeneities. Then, we write
\begin{equ}\label{Eq:RfPos}
\langle \cR f,\eta^\lambda_x \rangle = \langle \cR f - \Pi_x f(x),\eta^\lambda_x \rangle + \langle \Pi_x f(x),\eta^\lambda_x \rangle\;.
\end{equ}
It is easy to check that the first bound of (\ref{Eq:BoundDefBesovPos}) is satisfied. Regarding the second bound, it is satisfied by the first term on the right hand side of (\ref{Eq:RfPos}) thanks to the reconstruction bound (\ref{Eq:BoundRecons}). To show that the second term also satisfies the required bound, we distinguish two cases. Either we work with the polynomial regularity structure and then $\langle \Pi_x f(x),\eta^\lambda_x \rangle = 0$ as soon as $\eta$ kills polynomials. Or the lowest level with non-integer homogeneity is $\alpha$ and in that case
\begin{equ}
\bigg\| \Big\| \frac{\langle \Pi_x f(x),\eta^\lambda_x \rangle}{\lambda^{\bar\alpha}} \Big\|_{L^p} \bigg\|_{L^q_\lambda} \lesssim \Big\| |f(x)|_\alpha \Big\|_{L^p}\;,
\end{equ}
for all $\bar\alpha < \alpha$, as required.\\

\noindent\textit{Step 4: two models.} In the case where we deal with two models, the bounds above can be easily adapted in order to establish (\ref{Eq:BoundRecons2}). For instance, using obvious notations for elements built from the second model $(\bar{\Pi},\bar{\Gamma})$, we have
\begin{equs}
{}&A^{n+1}_y-\langle \Pi_{x} f(x),\phi^{n+1}_y \rangle - \bar{A}^{n+1}_y-\langle \bar{\Pi}_{x} \bar{f}(x),\phi^{n+1}_y \rangle \\
&= \int_{z\in B(y,2^{-(n+1)})}\!\!\!\! 2^{(n+1)|\s|} \langle \Pi_y \Gamma_{y,z}\big(f(z)-\Gamma_{z,x}f(x)\big)\\
&\qquad\qquad\qquad\qquad- \bar{\Pi}_y \bar{\Gamma}_{y,z}\big(\bar{f}(z)-\bar{\Gamma}_{z,x}\bar{f}(x)\big),\varphi^{n+1}_y\rangle dz\;.
\end{equs}
Then, we write
\begin{equs}\label{Eq:Decomp2models}
{}&\Pi_y \Gamma_{y,z}\big(f(z)-\Gamma_{z,x}f(x)\big)- \bar{\Pi}_y \bar{\Gamma}_{y,z}\big(\bar{f}(z)-\bar{\Gamma}_{z,x}\bar{f}(x)\big)\\
&= \Pi_y \Gamma_{y,z}\big(f(z)-\Gamma_{z,x}f(x)-\bar{f}(z)+\bar{\Gamma}_{z,x}\bar{f}(x)\big)\\
&\quad+(\Pi_y-\bar{\Pi}_y) \Gamma_{y,z}\big(\bar{f}(z)-\bar{\Gamma}_{z,x}\bar{f}(x)\big)+\bar{\Pi}_y(\Gamma_{y,z}- \bar{\Gamma}_{y,z})\big(\bar{f}(z)-\bar{\Gamma}_{z,x}\bar{f}(x)\big)\;,
\end{equs}
and the bound follows from the same arguments as in the case with a single model.\\

\noindent\textit{Step 5: uniqueness.} Finally, let us prove the uniqueness of $\cR f$. Let $\xi_1,\xi_2$ be two distributions satisfying the bound (\ref{Eq:BoundRecons}), and let $\rho\in\BB^r$ be an even function that integrates to $1$. For any compactly supported, smooth function $\psi:\R^d\rightarrow\R$ and for any $\delta\in(0,1]$, set
\begin{equ}
\psi_\delta(y) = \langle \rho^\delta_y,\psi\rangle = \int \rho^\delta(x-y)\psi(x)dx\;.
\end{equ}
Let $n\geq 0$ be the largest integer such that $2^{-n} \geq \delta$. Then, we get
\begin{equs}
\frac{\big|\langle \xi_1-\xi_2,\psi_\delta\rangle\big|}{\delta^\gamma} &\lesssim \|\psi\|_\infty \Big(\int_x \Big(\frac{\big|\langle \xi_1-\xi_2,\rho^\delta_x\rangle \big|}{\delta^\gamma}\Big)^p dx\Big)^{\frac1{p}}\\
&\lesssim \|\psi\|_\infty \int_{\lambda=2^{-(n+1)}}^{2^{-n}} \Big(\int_x \sup_{\eta\in\BB^r} \Big(\frac{\big|\langle \xi_1-\xi_2,\eta^\lambda_x\rangle \big|}{\lambda^\gamma}\Big)^p dx\Big)^{\frac1{p}} \frac{d\lambda}{\lambda \ln 2}\\
&\lesssim \|\psi\|_\infty \bigg(\int_{\lambda=2^{-(n+1)}}^{2^{-n}} \Big(\int_x \sup_{\eta\in\BB^r} \Big(\frac{\big|\langle \xi_1-\xi_2,\eta^\lambda_x\rangle \big|}{\lambda^\gamma}\Big)^p dx\Big)^{\frac{q}{p}} \frac{d\lambda}{\lambda \ln 2}\bigg)^{\frac1{q}}\;,
\end{equs}
which goes to $0$ as $n\rightarrow\infty$, or equivalently as $\delta\rightarrow 0$, thanks to (\ref{Eq:BoundRecons}). Consequently, $\langle \xi_1-\xi_2,\psi_\delta\rangle$ vanishes when $\delta\rightarrow 0$. On the other hand, $\langle \xi_1-\xi_2,\psi_\delta\rangle$ converges to $\langle \xi_1-\xi_2,\psi\rangle$ as $\delta\rightarrow 0$. We deduce that $\langle \xi_1-\xi_2,\psi\rangle = 0$ for all compactly supported, smooth functions $\psi$, and therefore, $\xi_1=\xi_2$.
\end{proof}

\section{Embedding theorems}\label{Sec:Embeddings}

The spaces $\cD^\gamma_{p,q}$ enjoy embedding properties which are similar to the well-known embeddings of Besov spaces, see for instance the book of Triebel~\cite[Sec. 2.3.2 and 2.7.1]{Triebel}. Recall that we work under Assumption \ref{Assumption:Gamma}, so that all the $\gamma,\gamma'$ below are implicitly assumed not to lie in $\cA$.

We say that we are in the \textit{periodic case} when the model for our regularity structure is periodic on some torus of $\R^d$ in the sense of \cite[Def.~3.33]{Hairer2014}. Implicitly, we then restrict the spaces $\cD^\gamma$ to elements $f$ which satisfy the same periodicity and we restrict the integrals (in $x$) in the norms to one period.

\begin{theorem}\label{Th:Embedding}
Let $(\cA,\cT,\cG)$ be a regularity structure and $(\Pi,\Gamma)$ be a model. Let $p,p',q,q' \in [1,\infty]$ and $\gamma,\gamma' > 0$ be a collection of parameters. The space $\cD^\gamma_{p,q}$ is continuously embedded into $\cD^{\gamma'}_{p',q'}$ in any of the following settings:\begin{enumerate}
\item $q' > q$, $p'=p$ and $\gamma'=\gamma$,
\item $q' \le q$, $p'=p$ and $\gamma' < \gamma$,
\item $q'=q$, $p'<p$ and $\gamma'=\gamma$ in the periodic case,
\item $q'=q$, $p'>p$ and $\gamma' < \gamma - |\s|\big(\frac1{p} - \frac1{p'}\big)$.
\end{enumerate}
\end{theorem}
\begin{remark}
Actually, our proof will show a slightly stronger statement in the fourth case: if $\big[\gamma - |\s|\big(\frac1{p} - \frac1{p'}\big),\gamma\big) \cap \cA = \emptyset$ then one can take $\gamma' =  \gamma - |\s|\big(\frac1{p} - \frac1{p'}\big)$.
\end{remark}

\subsection{A consequence of the embedding theorem}\label{Subsec:PAM}

Let us explain how this theorem allows to recover the expected H\"older regularity of the solution to the parabolic Anderson model \eqref{Eq:PAM} on $\R^3$ that we constructed in~\cite{mSHE}.

The specificity of that construction was twofold. First, the space of modelled distributions considered therein micmics locally in space the $\cB_{p,\infty}$ space with $p$ close to $1$: as explained in the introduction, taking $p$ close to $1$ instead of equal to $\infty$ provides a much simpler framework for starting the equation from a Dirac. The definition of modelled distributions we opted for in that work is slightly different from (and also less natural than) the definition presented here. However, it can be checked that the whole construction would carry through with the present definition of modelled distributions. Second, the space of modelled distributions are weighted at infinity in space: this is a consequence of the fact that the H\"older norm of the white noise blows up on unbounded spaces, so that the H\"older norm of the solution cannot be bounded either. However, the weighted spaces of modelled distributions are locally identical to their unweighted versions: this does not have any influence on the H\"older regularity of the solution so we can disregard this aspect of the spaces in our analysis below.

Since the roughest term in the expansion of the solution of the parabolic Anderson model is H\"older $1/2-$ in space, the reconstruction theorem shows that the solution constructed in~\cite{mSHE} belongs, as a function of the space variable, to the classical Besov space $\cB^{1/2-}_{p,\infty}$. The latter space can be embedded (using the classical embedding theorem) into $\cB^{\beta}_{\infty,\infty}$ for any $\beta < 1/2 - 3/p$: this is far from the expected H\"older $1/2-$ regularity that one expects for the solution.

On the other hand, if we apply our embedding theorem at the level of the spaces of modelled distributions, we no longer play with the regularity of the roughest term in the expansion of the solution, but rather with the order of the expansion that we denote by $\gamma$. It turns out that this parameter can be taken as large as desired in the construction, so that Theorem \ref{Th:Embedding} yields an element in $\cD^{\gamma'}_{\infty,\infty}$ with $\gamma' = \gamma - 3/p > 0$. The reconstruction theorem applied to the latter space yields an element of $\cB^{1/2-}_{\infty,\infty}$, thus proving the H\"older $1/2-$ regularity of the solution.

\subsection{Proof of the embedding theorem}

Before we proceed to the proof of the theorem, we state an elementary lemma.

\begin{lemma}\label{Lemma:RegInt}
Let $1\leq p \leq \tilde{p} \leq \infty$ and $0 < \delta \leq \tilde{\delta}$ be such that
\begin{equ}
\tilde{\delta} \leq \delta - |\s| \Big(\frac1{p} - \frac1{\tilde{p}}\Big)\;.
\end{equ}
For all $n\geq 0$ and all $u:\Lambda_n \rightarrow \R$, we have the bound
\begin{equ}
\Big\| \frac{u(x)}{2^{-n\tilde\delta}} \Big\|_{\ell^{\tilde{p}}_n} \leq \Big\| \frac{u(x)}{2^{-n\delta}} \Big\|_{\ell^{p}_n}\;.
\end{equ}
\end{lemma}
\begin{proof}
Assume that the right hand side is finite, otherwise there is nothing to prove. Then, we have
\begin{equ}
\sup_{x\in\Lambda_n} \Big|\frac{u(x)}{2^{-n\delta}}\Big| \leq 2^{n\frac{|\s|}{p}} \Big\| \frac{u(x)}{2^{-n\delta}} \Big\|_{\ell^{p}_n}\;.
\end{equ}
Therefore, we have
\begin{equs}
\Big\| \frac{u(x)}{2^{-n\tilde\delta}} \Big\|_{\ell^{\tilde{p}}_n}^{\tilde{p}} &\leq \Big\| \frac{u(x)}{2^{-n\delta}} \Big\|_{\ell^{p}_n}^p \sup_{x\in\Lambda_n} \Big|\frac{u(x)}{2^{-n\delta}}\Big|^{\tilde{p}-p} \frac1{2^{-n\tilde{p}(\tilde{\delta}-\delta)}} \\
&\leq \Big\| \frac{u(x)}{2^{-n\delta}} \Big\|_{\ell^{p}_n}^{\tilde{p}} 2^{n(\frac{\tilde{p}}{p}-1)|\s|+n\tilde{p}(\tilde{\delta}-\delta)}\;.
\end{equs}
By assumption, $\big(\frac{\tilde{p}}{p}-1\big)|\s|+\tilde{p}(\tilde{\delta}-\delta)\leq0$ and the asserted bound follows.
\end{proof}

\begin{proof}[of Theorem \ref{Th:Embedding}]
The cases 1, 2 and 4 will be established at the level of the spaces $\bar{\cD}$: this is sufficient since Theorem~\ref{Th:EqDgamma} then implies the same for the spaces $\cD$.

\medskip\noindent\textit{First case.} At the level of the spaces $\bar{\cD}$, this embedding is a direct consequence of the continuous embedding of $\ell^{q'}(\N)$ into $\ell^q(\N)$ whenever $q' > q$.

\medskip\noindent\textit{Second case.}  It suffices to show that the translation and consistency bounds hold upon replacing $q$ by $q'$, $\gamma$ by $\gamma'$ and $\bar{f}$ by its restriction to $\cT_{<\gamma'}$. The proof is exactly the same for the translation and the consistency bounds, so we only present the details for the former. First observe that for all $\epsilon > 0$ and all $1\leq q' \le q \leq \infty$, by H\"older's inequality we have
\begin{equ}
\Big(\sum_{n\geq 0}\sum_{h\in\cE_n} \Big(\frac{|u_n(h)|}{2^{n\epsilon}}\Big)^{q'}\Big)^{\frac1{q'}} \leq \Big(\sum_{n\geq 0}\sum_{h\in\cE_n} |u_n(h)|^{q}\Big)^{\frac1{q}} \Big(\sum_{n\geq 0} \sum_{h\in\cE_n} 2^{-n\epsilon\frac{q q'}{q-q'}}\Big)^{\frac{q-q'}{qq'}}\;.
\end{equ}
Then, we fix $\zeta \in \cA_{\gamma'}$. For any $n\geq 0$ and any $h\in\cE_n$, we set
\begin{equ}
u_n(h) = \bigg\| \frac{\big| \aver{f}{n}(x+h)-\Gamma_{x+h,x} \aver{f}{n}(x) \big|_\zeta}{2^{-n(\gamma-\zeta)}} \bigg\|_{\ell^p_n}\;.
\end{equ}
Applying the above inequality with $\epsilon = \gamma-\gamma'$, one immediately gets
\begin{equ}\label{Eq:qq'}
\Bigg(\sum_{n\geq 0}\sum_{h\in\cE_n} \bigg\| \frac{\big| \aver{f}{n}(x+h)-\Gamma_{x+h,x} \aver{f}{n}(x) \big|_\zeta}{2^{-n(\gamma'-\zeta)}} \bigg\|_{\ell^p_n}^{q'}\Bigg)^{\frac1{q'}} \lesssim \$ f\$\;.
\end{equ}
This yields the desired embedding in the case where $\gamma' \in (\max \cA_\gamma,\gamma)$. If $\gamma' < \max \cA_\gamma$, the restriction of $\aver{f}{n}$ to $\cT_{<\gamma'}$ differs from $\aver{f}{n}$ by $\sum_{\gamma' < \beta < \gamma} \aver{f}{n}_\beta$. Then, uniformly over all $n\geq 0$, we have
\begin{equs}
{}&\Bigg(\sum_{n\geq 0}\sum_{h\in\cE_n} \bigg\| \frac{\big| \Gamma_{x+h,x}\sum_{\gamma'<\beta < \gamma} \aver{f}{n}_\beta(x) \big|_\zeta}{2^{-n(\gamma'-\zeta)}} \bigg\|_{\ell^p_n}^{q'}\Bigg)^{\frac1{q'}}\\
&\lesssim \sum_{\gamma' < \beta < \gamma}  \Bigg(\sum_{n\geq 0} \Big\| \big|\aver{f}{n}(x)\big|_\beta 2^{-n(\beta-\gamma')} \Big\|_{\ell^p_n}^{q'}\Bigg)^{\frac1{q'}}\\
&\lesssim \sup_{\beta\in\cA_\gamma}\sup_{n\geq 0} \Big\| \big| \aver{f}{n} (x) \big|_\beta \Big\|_{\ell^p_n}\;,
\end{equs}
which implies, together with (\ref{Eq:qq'}), the desired embedding at the level of $\bar\cD$.

\medskip\noindent\textit{Third case.} This follows from the continuous embedding of $L^p$ into $L^{p'}$ whenever the underlying space is bounded.

\medskip\noindent\textit{Fourth case.} Here again, we prove the embedding at the level of the spaces $\bar\cD$. Let $\zeta_1 > \zeta_2 > \ldots$ be the elements of $\cA_{\gamma}$ in the decreasing order. For all $\epsilon \geq 0$ and all $\zeta<\gamma$, we let $p_\zeta^{(\epsilon)} \in [p,\infty]$ be such that
\begin{equ}
\zeta + \epsilon = \gamma - |\s| \Big(\frac1{p} - \frac1{p_\zeta^{(\epsilon)}}\Big)\;,
\end{equ}
if this admits a solution, otherwise we set $p_\zeta^{(\epsilon)}=\infty$. In the particular case $\epsilon = 0$, we set $p_\zeta=p_\zeta^{(0)}$.

The core of the proof relies on the following two properties:\begin{enumerate}
\item[1)] For all $\gamma'\in (\zeta_1,\gamma)$, let $p''\in[p,\infty]$ be such that $\gamma' = \gamma - |\s| \Big(\frac1{p}-\frac1{p''}\Big)$ if this admits a solution, otherwise let $p''=\infty$. Then, $\bar{f} \in \bar{\cD}^{\gamma'}_{p'',q}$.
\item[2)] For all $\gamma'\in(\zeta_2,\zeta_1)$, we have $\bar{f} \in \bar{\cD}^{\gamma'}_{p_{\zeta_1},q}$.
\end{enumerate}
Once these properties are established, an elementary recursion on $\zeta_i$ concludes the proof of the embedding. We are left with proving these two properties.

Property 1) is a direct consequence of Lemma \ref{Lemma:RegInt} and of the continuous inclusion $\ell^p_0 \subset \ell^{p''}_0$. To prove Property 2), we actually show:\begin{enumerate}
\item[2')] For all $\epsilon >0$ and all $\gamma''\in(\zeta_2,\zeta_1)$, we have $\bar{f} \in \bar{\cD}^{\gamma''}_{p_{\zeta_1}^{(\epsilon)},q}$.
\end{enumerate}
Indeed, if 2') is satisfied then Lemma \ref{Lemma:RegInt} and the continuous inclusion $\ell^{p_{\zeta_1}^{(\epsilon)}}_0 \subset \ell^{p_{\zeta_1}}_0$ ensure that $\bar{f} \in \bar{\cD}^{\gamma''(\epsilon)}_{p_{\zeta_1},q}$ where $\gamma''(\epsilon) = \gamma'' - |\s| \Big(\frac1{p_{\zeta_1}^{(\epsilon)}}-\frac1{p_{\zeta_1}}\Big)$. Since $\gamma''(\epsilon) \uparrow \gamma''$ as $\epsilon\downarrow 0$, we deduce that Property 2) holds.

Let us show 2'). Fix $\gamma''\in(\zeta_2,\zeta_1)$. Let $\bar{f}_{<\zeta_1}$ be the restriction of $\bar{f}$ to $\cT_{<\zeta_1}$. For all $\zeta < \zeta_1$, we have
\begin{equs}
\big| \aver{f}{n}_{<\zeta_1}(x+h)-\Gamma_{x+h,x} \aver{f}{n}_{<\zeta_1}(x) \big|_\zeta &\leq \big| \aver{f}{n}(x+h)-\Gamma_{x+h,x} \aver{f}{n}(x) \big|_\zeta\\
&\quad+ \big|\Gamma_{x+h,x} \aver{f}{n}_{\zeta_1}(x) \big|_\zeta\;.
\end{equs}
Since $\gamma'' < \zeta_1 +\epsilon$, Lemma \ref{Lemma:RegInt} yields the bound
\begin{equ}
\bigg(\sum_{n\geq 0} \sum_{h\in \cE_n} \bigg\| \frac{\big| \aver{f}{n}(x+h)-\Gamma_{x+h,x} \aver{f}{n}(x) \big|_\zeta}{2^{-n(\gamma''-\zeta)}} \bigg\|^q_{\ell^{p_{\zeta_1}^{(\epsilon)}}_n} \bigg)^{\frac1{q}} \lesssim \$\bar{f}\$\;.
\end{equ}
On the other hand, we have
\begin{equs}
\bigg(\sum_{n\geq 0} \sum_{h\in \cE_n} \bigg\| \frac{\big|\Gamma_{x+h,x} \aver{f}{n}_{\zeta_1}(x) \big|_\zeta}{2^{-n(\gamma''-\zeta)}} \bigg\|^q_{\ell^{p_{\zeta_1}^{(\epsilon)}}_n} \bigg)^{\frac1{q}} &\lesssim \sup_{n\geq 0} \|\aver{f}{n}_{\zeta_1}\|_{\ell^{p_{\zeta_1}^{(\epsilon)}}_{n}} \Big(\sum_{n\geq 0} 2^{-n(\zeta_1-\gamma'')q}\Big)^{\frac1{q}}\\
&\lesssim \sup_{n\geq 0} \|\aver{f}{n}_{\zeta_1}\|_{\ell^{p_{\zeta_1}^{(\epsilon)}}_{n}}\;,
\end{equs}
so we only need to bound this last term. A careful inspection of the proof of 
Lemma~\ref{Lemma:LocalBoundn} shows that there exists $C>0$ such that
\begin{equ}
\|\aver{f}{n+1}_{\zeta_1}\|_{\ell^{p_{\zeta_1}^{(\epsilon)}}_{n+1}} \leq \|\aver{f}{n}_{\zeta_1}\|_{\ell^{p_{\zeta_1}^{(\epsilon)}}_{n}} + C 2^{-n\epsilon} \sum_{h\in \cE_{n+1}} \Big\| \frac{\big|\aver{f}{n} (x) - \Gamma_{x,x+h}\aver{f}{n+1}(x+h)\big|_{\zeta_1}}{2^{-n\epsilon}} \Big\|_{\ell^{p_{\zeta_1}^{(\epsilon)}}_n}\;,
\end{equ}
uniformly over all $n\geq 0$. Since $\zeta_1+\epsilon\leq \gamma-|\s|(\frac1{p}-\frac1{p_{\zeta_1}^{(\epsilon)}})$, Lemma \ref{Lemma:RegInt} yields the bound
\begin{equ}
\Big\| \frac{\big|\aver{f}{n} (x) - \aver{f}{n+1}(x+h)\big|_{\zeta_1}}{2^{-n\epsilon}} \Big\|_{\ell^{p_{\zeta_1}^{(\epsilon)}}_n} \leq \Big\| \frac{\big|\aver{f}{n} (x) - \Gamma_{x,x+h}\aver{f}{n+1}(x+h)\big|_{\zeta_1}}{2^{-n(\gamma-\zeta_1)}} \Big\|_{\ell^p_n}\;,
\end{equ}
so that, using Remark \ref{Rmk:Consistency}, we deduce that
\begin{equ}
\|\aver{f}{n+1}_{\zeta_1}\|_{\ell^{p_{\zeta_1}^{(\epsilon)}}_{n+1}} \leq \|\aver{f}{n}_\zeta\|_{\ell^{p_{\zeta_1}^{(\epsilon)}}_{n}} + C 2^{-n\epsilon} \$\bar{f}\$\;,
\end{equ}
for all $n\geq 0$. Consequently,
\begin{equ}
\sup_{n\geq 0} \|\aver{f}{n}_{\zeta_1}\|_{\ell^{p_{\zeta_1}^{(\epsilon)}}_{n}} \lesssim \$\bar{f}\$\;,
\end{equ}
thus concluding the proof of 2').
\end{proof}

\section{Convolution with singular kernels}

From now on, we assume that $d\geq 2$ and we view the first direction $x_1$ as time and the $d-1$ remaining ones as space. We consider a kernel $P:\R^d \rightarrow \R$ which is smooth except at the origin and which improves regularity by order $\beta >0$. More precisely, we assume that there exist some smooth functions $P_-$ and $P_n$, $n\geq 0$ on $\R^d$ such that:
\begin{enumerate}
\item For every $x\in\R^d\backslash\{0\}$, we have $P(x) = P_-(x) + \sum_{n\geq 0} P_n(x)$,
\item $P_0$ is supported in $B(0,1)$ and for every $n\geq 0$, we have the identity
\begin{equ}
P_n(x) = 2^{n(|\s|-\beta)} P_0(2^{n\s}x)\;,\quad x\in\R^d\;,
\end{equ}
where $2^{n\s}x = (2^{n\s_1}x_1,\ldots,2^{n\s_d}x_d)$.
\item The function $P_0$ annihilates all polynomials of scaled degree $r$.
\end{enumerate}

The second property ensures that for all $k\in\N^d$, there exists $C>0$ such that
\begin{equ}\label{Eq:ScalingPn}
\big| \partial^k P_n(x) \big| \leq C 2^{n(|\s|-\beta + |k|)}\;,
\end{equ}
uniformly over all $n\geq 0$ and all $x\in\R^d$.

A typical example of such a kernel is given by the heat kernel which, under the parabolic scaling $\s = (2,1,\ldots,1)$, satisfies these assumptions with $\beta=2$.
The celebrated Schauder estimates assert that for any distribution $\xi\in\cB^\alpha_{p,q}$ (that does not grow too fast at infinity), the distribution $P*\xi$ obtained by convolving $\xi$ with $P$ belongs to $\cB^{\alpha+\beta}_{p,q}$. Notice that the growth condition is only required because the kernel is not compactly supported. The core of the proof of the Schauder estimate concerns the convolution with the singular part $P_+ = \sum_{n\geq 0} P_n$ of the kernel for which no growth condition is required. These Schauder estimates are crucial for solving (stochastic) PDEs as they allow to obtain a fixed point.

The goal of the present section is to lift this result to our spaces $\cD^\gamma_{p,q}$. Actually, we will restrict ourselves to proving this result with $P$ replaced by $P_+$ since this is the only difficult part in the proof. The abstract convolution operator was defined in \cite[Sec.~4]{Hairer2014} in the H\"older setting, and the extension to the more general Besov setting does not present any major difference: in particular, the definition of the operator is formally the same. As we explained in the introduction, in the theory of regularity structures a function/distribution is described through a collection of generalised Taylor expansions, at each space-time point, on a given basis of monomials. Recall that this basis contains two types of elements: classical space-time monomials; and abstract monomials which are built from the driving noise. When convolving with a singular kernel, one expects two types of terms: classical space-time monomials, and abstract monomials obtained by convolving with the kernel the original ones. To give a precise meaning to the latter, as in \cite[Sec.~4]{Hairer2014}, we assume in this section that our regularity structure is equipped with an abstract integration map of order $\beta$, namely a linear map $\cI : \cT\rightarrow\cT$ such that:
\begin{enumerate}
\item $\cI : \cT_\zeta \rightarrow \cT_{\zeta+\beta}$,
\item $\cI \tau = 0$ for all $\tau \in \bar\cT$,
\item $\cI \Gamma \tau - \Gamma \cI \tau \in \bar\cT$ for all $\tau \in \cT$ and all $\Gamma\in\cG$.
\end{enumerate}
Second, we assume that our model is admissible in the sense that it satisfies the identity:
\begin{equ}[e:defPix]
\Pi_x \cI \tau(y) = \langle \Pi_x \tau , P_+(y-\cdot) \rangle - \sum_{k\in\N^d:|k| < \zeta +\beta} \frac{X^k}{k!} \langle \Pi_x \tau , \partial^k P_+(x-\cdot) \rangle\;,
\end{equ}
for all $\tau\in\cT_\zeta$ and all $\zeta \in \cA_\gamma$. (See again \cite[Sec.~4]{Hairer2014} for
a discussion of the meaning of this condition.)
Then, we introduce the linear operator
\begin{equs}
\cP^\gamma_+ f(x) &:= \cI(f(x)) + \sum_{\zeta\in\cA_\gamma} \sum_{k\in\N^d:|k|_\s < \zeta +\beta} \frac{X^k}{k!} \langle \Pi_x \cQ_\zeta f(x) , \partial^k P_+(x-\cdot) \rangle\\
&+ \sum_{k\in\N^d:|k|_\s < \gamma +\beta} \frac{X^k}{k!} \langle \cR f - \Pi_x f(x) , \partial^k P_+(x-\cdot)\rangle\;.
\end{equs}
Let us describe informally the three terms appearing in this expression. The first term takes values in the 
non-classical part of the regularity structure. Looking at \eqref{e:defPix}, we see that the corresponding
local expansion is obtained from that of $f$ itself by convolving each term with the singular kernel and 
then subtracting their classical Taylor expansions (at lower integer levels). Adding these Taylor expansions 
back yields the second term. The third term is the only non-local term in the definition of the operator: 
as we will see in the proof of the theorem below, this is precisely the required quantity ensuring that 
reconstruction and convolution commute and that a Schauder estimate holds.
\begin{theorem}\label{Th:Conv}
Consider a regularity structure equipped with an integration map of order $\beta > 0$ and
an admissible model.
Fix $\gamma \in \R_+\backslash\N$, and assume that $\gamma+\beta \notin\N$. Then $\cP^\gamma_+$ is a continuous linear map from $\cD^{\gamma}_{p,q}$ into $\cD^{\gamma+\beta}_{p,q}$ and we have the identity
\begin{equ}\label{Eq:ConvolRecons}
\cR \cP^\gamma_+ f = P_+ * \cR f\;,
\end{equ}
for all $f\in\cD^{\gamma}_{p,q}$. Furthermore, if $(\bar{\Pi},\bar{\Gamma})$ is another admissible model, then we have
\begin{equs}
\$ \cP^\gamma_+ f, \cP^\gamma_+ \bar{f} \$ \lesssim \|\Pi\| (1 + \|\Gamma\|) \$f,\bar{f}\$ + \big(\|\Pi-\bar{\Pi}\|(1+\|\bar{\Gamma}\|) + \|\bar{\Pi}\| \|\Gamma-\bar{\Gamma}\|\big)\$\bar{f}\$\;,
\end{equs}
uniformly over all models $(\Pi,\Gamma)$, $(\bar{\Pi},\bar{\Gamma})$, and all elements $f,\bar{f}$ in $\cD^\gamma_{p,q}$, $\bar{\cD}^\gamma_{p,q}$.
\end{theorem}

Before we proceed to the proof, we introduce a few notations and state a useful lemma. We set $\gamma' := \gamma + \beta$. For all $k\in\N^d$, we set
\begin{equ}
P^{k,\gamma'}_{n,x,y}(\cdot) := \partial^k P_n(y-\cdot) - \sum_{\ell\in\N^d:|k+\ell|<\gamma'} \frac{(y-x)^\ell}{\ell!} \partial^{k+\ell} P_n(x-\cdot)\;,
\end{equ}
and $P^{k,\gamma'}_{x,y} = \sum_{n\geq 0} P^{k,\gamma'}_{n,x,y}$. Let $e_i$ be the unit vector of $\R^d$ in the direction $i\in\{1,\ldots,d\}$, and for every $\ell\in\N^d$ set $\m(\ell):=\inf\{i: \ell_i \ne 0\}$. We define
\begin{equ}
\partial \gamma' := \big\{\ell\in\N^d: |\ell|_\s >\gamma', |\ell-e_{\m(\ell)}|_\s < \gamma'\big\} \;.
\end{equ}
We then recall the following identity.
\begin{lemma}[Prop 11.1~\cite{Hairer2014}]\label{Lemma:TaylorSpaceTime}
For all $x,y\in\R^d$ and all $k\in\N^d$ such that $|k|_\s<\gamma'$, we have
\begin{equ}
P_{n,x,y}^{k,\gamma'}(\cdot) = \sum_{\ell:k+\ell \in \partial \gamma'} \int_{\R^d} \partial^{k+\ell} P_n(x+h-\cdot) \mu^{\ell}(y-x,dh)\;.
\end{equ}
Here, $\mu^{\ell}(y-x,dh)$ is a signed measure on $\R^d$, supported in the set $\{z\in\R^d: z_i \in [0,y_i-x_i]\}$ and whose total mass is given by $\frac{(y-x)^{\ell}}{\ell!}$.
\end{lemma}

\begin{proof}[of Theorem \ref{Th:Conv}] 
We start with the local bound of the $\cD^\gamma_{p,q}$-norm. For every $\zeta \in \cA_{\gamma'}\backslash\N$, the only contributions of $\cP_+^\gamma f$ at level $\zeta$ come from $\cI(f(x))$ and we have
\begin{equ}
\big\| \cQ_{\zeta+\beta}\cI(f(x)) \big\|_{L^p} \lesssim \big\| \cQ_{\zeta} f(x) \big\|_{L^p}\;,
\end{equ}
by the properties of $\cI$. Let us now consider $k\in\N^d$ such that $|k|_\s<\gamma'$. We have the identity
\begin{equs}\label{Eq:LocalBoundKernel}
k! \cQ_k \cP^\gamma_+ f(x) = \sum_{n\geq 0} \Big( &\sum_{\zeta\in\cA_\gamma: \zeta > |k|_\s-\beta} \langle \Pi_x \cQ_\zeta f(x) , \partial^k P_n(x-\cdot) \rangle\\
& + \langle \cR f - \Pi_x f(x) , \partial^k P_n(x-\cdot) \rangle\Big) \;.
\end{equs}
By (\ref{Eq:ScalingPn}), we have
\begin{equs}
\Big\| \langle \Pi_x \cQ_\zeta f(x) , \partial^k P_n(x-\cdot) \rangle\Big\|_{L^p} \lesssim \big\| \cQ_\zeta f(x) \big\|_{L^p} 2^{-n(\zeta+\beta-|k|)}\;,
\end{equs}
uniformly over all $n\geq 0$. The sum over all $n\geq 0$ of these norms is therefore bounded by a term of order $\$f\$$ as required. Applying Theorem \ref{Th:Reconstruction}, we obtain a similar bound for the second term on the right hand side of (\ref{Eq:LocalBoundKernel}).

We turn to the translation bound. Regarding the terms at non-integer levels, the bound derives from exactly the same argument as for the local bound. We focus on terms at integer levels. Let $k\in\N^d$ such that $|k|_\s < \gamma +\beta$. A simple computation based on~\cite[Lemma 5.16]{Hairer2014} ensures that
\begin{equs}\label{Eq:TranslationBoundKernel1}
{}&k!\cQ_k\big( \cP^\gamma_+ f(x+h) - \Gamma_{x+h,x} \cP^\gamma_+ f(x) \big)\\
&= \big\langle \cR f - \Pi_x f(x) , P^{k,\gamma'}_{x,x+h}\big\rangle\\
&\quad- \sum_{\substack{\zeta\in\cA_{\gamma}\\\zeta \leq |k|_\s-\beta}} \big\langle \Pi_{x+h}\cQ_\zeta\big(f(x+h) - \Gamma_{x+h,x} f(x)\big) , \partial^k P(x+h-\cdot)\big\rangle\;,
\end{equs}
which can also be written as
\begin{equs}\label{Eq:TranslationBoundKernel2}
{}&k!\cQ_k\big( \cP_+^\gamma f(x+h) - \Gamma_{x+h,x} \cP_+^\gamma f(x) \big)\\
&= \big\langle \cR f - \Pi_{x+h} f(x+h) , \partial^k P_+(x+h-\cdot)\big\rangle\\
&\quad-\big\langle \cR f - \Pi_{x} f(x) , \sum_{\ell\in\N^d:|k+\ell|_\s < \gamma'}\frac{h^\ell}{\ell!} \partial^{k+\ell} P_+(x-\cdot)\big\rangle\\
&\quad+ \sum_{\substack{\zeta\in\cA_{\gamma}\\\zeta > |k|_\s-\beta}} \big\langle \Pi_{x+h}\cQ_\zeta\big(f(x+h) - \Gamma_{x+h,x} f(x)\big) , \partial^k P_+(x+h-\cdot)\big\rangle\;.
\end{equs}
Let $n_0$ be the largest integer such that $2^{-n_0} \geq |h|$. According to the relative values of $n$ and $n_0$ we use either of these two expressions for the proof of the bound. We start with the case $n<n_0$. We have
\begin{equs}
{}&\Big\| \sum_{n<n_0}\frac{\langle \cR f - \Pi_x f(x) , P^{k,\gamma'}_{n,x,x+h}\rangle}{\|h\|_\s^{\gamma'-|k|_\s}} \Big\|_{L^p}\\
&\lesssim \sum_{n<n_0} \sum_{\ell \in \partial \gamma'}\frac{2^{-n(\gamma'-|\ell|_\s)} \|h\|_\s^{|\ell - k|_\s}}{\|h\|_\s^{\gamma'-|k|_\s}}\Big\| \sup_{\eta\in\BB^r}\frac{\big|\langle \cR f - \Pi_x f(x) , \eta^{2^{-n}}_x\rangle\big|}{2^{-n\gamma}} \Big\|_{L^p}\;,
\end{equs}
uniformly over all $h\in B(0,2^{-n_0})\backslash B(0,2^{-n_0-1})$ and all $n_0\geq 0$. Since
\begin{equ}
\sum_{n<n_0} \sum_{\ell \in \partial \gamma'}\frac{2^{-n(\gamma'-|\ell|_\s)} \|h\|_\s^{|\ell - k|_\s}}{\|h\|_\s^{\gamma'-|k|_\s}} \lesssim 1\;,
\end{equ}
uniformly over the same parameters, we get using Jensen's inequality that
\begin{equs}
{}&\bigg(\int_{h\in B(0,1)} \Big\| \sum_{n<n_0}\frac{\langle \cR f - \Pi_x f(x) , P^{k,\gamma'}_{n,x,x+h}\rangle}{\|h\|_\s^{\gamma'-|k|_\s}} \Big\|_{L^p}^q \frac{dh}{\|h\|_\s^{|\s|}} \bigg)^{\frac1{q}}\\
&\lesssim \bigg(\int_{h\in B(0,1)} \sum_{n<n_0}\sum_{\ell \in \partial \gamma'} 2^{-(n-n_0)(\gamma'-|\ell|_\s)}\\
&\qquad\qquad \times\Big\| \sup_{\eta\in\BB^r}\frac{\big|\langle \cR f - \Pi_x f(x) , \eta^{2^{-n}}_x\rangle\big|}{2^{-n\gamma}} \Big\|_{L^p}^q\frac{dh}{\|h\|_\s^{|\s|}} \bigg)^{\frac1{q}}\\
&\lesssim \bigg(\sum_{n\geq 0} \Big\| \sup_{\eta\in\BB^r}\frac{\big|\langle \cR f - \Pi_x f(x) , \eta^{2^{-n}}_x\rangle\big|}{2^{-n\gamma}} \Big\|_{L^p}^q\bigg)^{\frac1{q}}\;,
\end{equs}
which is bounded by a term of order $\$f\$$ by Theorem \ref{Th:Reconstruction}. The second term on the right hand side of (\ref{Eq:TranslationBoundKernel1}) can be bounded similarly.

We turn to the case $n\geq n_0$, and we use (\ref{Eq:TranslationBoundKernel2}). To bound the first term, we use a change of variable at the second line to get
\begin{equs}
{}&\Big\| \sum_{n\geq n_0}\frac{\langle \cR f - \Pi_{x+h} f(x+h) , \partial^k P_n(x+h-\cdot) \rangle}{|h|^{\gamma'-|k|}} \Big\|_{L^p}\\
&\lesssim \sum_{n\geq n_0} \Big\| \frac{\langle \cR f - \Pi_x f(x) , \partial^k P_n(x-\cdot) \rangle}{|h|^{\gamma'-|k|}} \Big\|_{L^p}\\
&\lesssim \sum_{n\geq n_0} \frac{2^{-n(\gamma'-|k|)}}{|h|^{\gamma'-|k|}} \Big\| \sup_{\eta\in\BB^r} \frac{\big|\langle \cR f - \Pi_x f(x) , \eta^{2^{-n}}_x \rangle\big|}{2^{-n\gamma}} \Big\|_{L^p}\;,
\end{equs}
uniformly over all $h\in B(0,2^{-n_0})\backslash B(0,2^{-n_0-1})$ and all $n_0\geq 0$. Since
\begin{equ}
\sum_{n\geq n_0} \frac{2^{-n(\gamma'-|k|)}}{|h|^{\gamma'-|k|}} \lesssim 1\;,
\end{equ}
uniformly over the same parameters, we apply Jensen's inequality to get
\begin{equs}
{}&\bigg(\int_{h\in B(0,1)} \Big\|\sum_{n\geq n_0}\frac{\langle \cR f - \Pi_{x+h} f(x+h) , \partial^k P_n(x+h-\cdot) \rangle}{|h|^{\gamma'-|k|}} \Big\|_{L^p}^q \frac{dh}{|h|^{|\s|}} \bigg)^{\frac1{q}}\\
&\lesssim \bigg(\int_{h\in B(0,1)}\sum_{n\geq n_0}\frac{2^{-n(\gamma'-|k|)}}{|h|^{\gamma'-|k|}} \Big\| \sup_{\eta\in\BB^r} \frac{\big|\langle \cR f - \Pi_x f(x) , \eta^{2^{-n}}_x \rangle\big|}{2^{-n\gamma}} \Big\|_{L^p}^q \frac{dh}{|h|^{|\s|}} \bigg)^{\frac1{q}}\\
&\lesssim \$f\$\;,
\end{equs}
as required. To bound the second and third terms arising from (\ref{Eq:TranslationBoundKernel2}), one proceeds similarly.

Let us now show that $\cR \cP^\gamma_+ f = P_+*\cR f$. By the uniqueness part of 
Theorem~\ref{Th:Reconstruction}, it suffices to show that
\begin{equs}
\bigg\| \Big\| \sup_{\eta \in \BB^r} \frac{\big|\langle P_+*\cR f - \Pi_x \cP^\gamma_+ f(x), \eta_x^\lambda\rangle \big|}{\lambda^{\gamma'}} \Big\|_{L^p} \bigg\|_{L^q_\lambda} < \infty\;.
\end{equs}
By the restriction property of the spaces $\cD^\gamma_{p,q}$ we can furthermore assume without loss of
generality that $\gamma' \in (0,1)$ which simplifies a number of expressions below.
We have
\begin{equs}
\langle P_+*\cR f - \Pi_x \cP^\gamma_+ f(x), \eta_x^\lambda\rangle &= \sum_{n\geq 0} \int_y \eta^\lambda_x(y) \Big(\big\langle \cR f , P_n(y-\cdot) \big\rangle- \langle \Pi_x f(x) , P_n(y-\cdot)\rangle\\
& - \langle \cR f - \Pi_x f(x) , P_n(x-\cdot)\rangle\Big)dy\\
&= \sum_{n\geq 0} \Big\langle \cR f - \Pi_x f(x) , \int_y \eta^\lambda_x(y) P^{0,\gamma'}_{n,x,y} dy \Big\rangle\;.
\end{equs}
Then, we argue differently according to the relative values of $2^{-n}$ and $\lambda$. Let $n_0$ be the largest integer such that $2^{-n_0} \geq \lambda$. If $n\leq n_0$, then $\int_y \eta^\lambda_x(y) P^{0,\gamma'}_{n,x,y}$ scales like a function $\psi^{2^{-n+1}}_x$ times a factor of order $2^{-n \beta}$, for some $\psi\in\BB^r$ (depending on $n$ and $y$), uniformly over all $n\leq n_0$. Therefore, we get
\begin{equs}
{}&\Big\| \sup_{\eta \in \BB^r} \sum_{n\leq n_0} \frac{\big|\langle \cR f - \Pi_x f(x) , \int_y \eta^\lambda_x(y) P^{\gamma,0}_{n,x,y-x} dy\rangle \big|}{\lambda^{\gamma'}} \Big\|_{L^p}\\
&\lesssim \sum_{n\leq n_0} \Big\| \sup_{\psi \in \BB^r} \frac{\big|\langle \cR f - \Pi_x f(x) , \psi^{2^{-n+1}}_x \rangle \big|}{2^{-n\gamma}} \Big\|_{L^p} 2^{-(n-n_0)\gamma'}\;,
\end{equs}
uniformly over all $n_0\geq 0$ and all $\lambda\in B(0,2^{-n_0})\backslash B(0,2^{-n_0-1})$. 
Jensen's inequality then yields
\begin{equs}
{}&\bigg\| \Big\| \sup_{\eta \in \BB^r} \sum_{n\leq n_0} \frac{\big|\langle \cR f - \Pi_x f(x) , \int_y \eta^\lambda_x(y) P^{0,\gamma}_{n,x,y} dy\rangle \big|}{\lambda^{\gamma'}} \Big\|_{L^p}\bigg\|_{L^q}\\
&\lesssim \bigg(\int_{\lambda\in B(0,1)} \sum_{n\leq n_0}2^{-(n-n_0)\gamma'} \Big\| \sup_{\psi \in \BB^r} \frac{\big|\langle \cR f - \Pi_x f(x) , \psi^{2^{-n}}_x \rangle \big|}{2^{-n\gamma}} \Big\|_{L^p}^q \frac{d\lambda}{\lambda}\bigg)^{\frac1{q}}\\
&\lesssim \bigg(\sum_{n\geq 0} \Big\| \sup_{\psi \in \BB^r} \frac{\big|\langle \cR f - \Pi_x f(x) , \psi^{2^{-n}}_x \rangle \big|}{2^{-n\gamma}} \Big\|_{L^p}^q \bigg)^{\frac1{q}}\;,
\end{equs}
which is of order $\$f\$$ as required. We turn to the case $n> n_0$. We bound separately the contributions coming from each of the two terms in $P^{\gamma',0}_{n,x,y}$. The function $\int_y \eta^\lambda_x(y) P_n(y-\cdot)dy$ scales like $\psi^{2\lambda}_x(\cdot)$ times a factor of order $2^{-\beta n}$, for some $\psi\in\BB^r$, uniformly over all $n\geq n_0$ and all $\lambda\in B(0,2^{-n_0})\backslash B(0,2^{-n_0-1})$. This being given, we have
\begin{equs}
{}&\bigg\| \Big\| \sup_{\eta \in \BB^r} \sum_{n> n_0} \frac{\big|\langle \cR f - \Pi_x f(x) , \int_y \eta^\lambda_x(y) P_n(y-\cdot) dy\rangle \big|}{\lambda^{\gamma'}} \Big\|_{L^p}\bigg\|_{L^q}\\
&\lesssim \bigg\| \sum_{n> n_0}2^{-\beta n}\Big\| \sup_{\psi \in \BB^r}  \frac{\big|\langle \cR f - \Pi_x f(x) , \psi^{2\lambda}_x \rangle \big|}{\lambda^{\gamma'}} \Big\|_{L^p}\bigg\|_{L^q}\\
&\lesssim \bigg\| \Big\| \sup_{\psi \in \BB^r}  \frac{\big|\langle \cR f - \Pi_x f(x) , \psi^{\lambda}_x \rangle \big|}{\lambda^{\gamma}} \Big\|_{L^p}\bigg\|_{L^q}\;,
\end{equs}
which is of order $\$f\$$ as required. The term with $P_n(y-\cdot)$ replaced by $P_n(x-\cdot)$
is bounded analogously, thus concluding the proof of (\ref{Eq:ConvolRecons}).

In the case where we deal with two models, the above arguments can be adapted, using the reconstruction bound (\ref{Eq:BoundRecons2}) as well as decompositions similar to what we did in (\ref{Eq:Decomp2models}).
\end{proof}

\bibliographystyle{Martin}
\bibliography{library_reconstruction}

\end{document}